\documentclass[a4paper,11pt,twoside]{article}
\usepackage[T1]{fontenc}
\usepackage[utf8]{inputenc}
\usepackage[english]{babel}
\usepackage{lmodern}
\usepackage{microtype}
\usepackage{geometry}
\usepackage{setspace}
\usepackage{booktabs}
\usepackage{graphicx}
\usepackage{hyperref}
\usepackage[nottoc]{tocbibind}
\usepackage{textgreek}
\usepackage{enumitem}
\usepackage{amsmath,amssymb}
\usepackage{amsthm}
\usepackage{amscd}
\usepackage{mathtools}
\usepackage{mathrsfs}
\usepackage{braket}
\usepackage{extarrows}
\usepackage{tikz}
\usepackage{tikz-cd}
\usetikzlibrary{arrows}

\theoremstyle{plain}
\newtheorem{thm}{Theorem}[section]
\newtheorem{cor}[thm]{Corollary}
\newtheorem{prp}[thm]{Proposition}
\newtheorem{lem}[thm]{Lemma}

\theoremstyle{definition}
\newtheorem{dfn}[thm]{Definition}
\newtheorem{rmk}[thm]{Remark}

\numberwithin{equation}{section}

\newcommand{\mi}{\mu}
\newcommand{\N}{\mathbb{N}}
\newcommand{\Z}{\mathbb{Z}}
\newcommand{\Q}{\mathbb{Q}}
\newcommand{\R}{\mathbb{R}}

\newcommand{\id}{\mathrm{id}}
\newcommand{\D}{\mathbb{D}}
\newcommand{\X}{\mathfrak{X}}
\newcommand{\I}{\mathfrak{I}}
\newcommand{\Hr}{\mathscr{H}}
\newcommand{\E}{\mathcal{E}}
\newcommand{\nr}{\mathrm{nr}}
\newcommand{\an}{\mathrm{an}}
\newcommand{\rig}{\mathrm{rig}}
\newcommand{\ad}{\mathrm{ad}}
\newcommand{\pdiv}{p\mathrm{-div}}
\newcommand{\pgr}{p\mathrm{-gr}}

\renewcommand{\phi}{\varphi}
\renewcommand{\epsilon}{\varepsilon}
\renewcommand{\O}{\mathcal{O}}
\renewcommand{\k}{\mathit{k}}
\renewcommand{\int}{\mathrm{int}}
\renewcommand{\H}{\mathcal{H}}
\renewcommand{\r}{\mathrm{r}}

\DeclareMathOperator{\Ker}{Ker}

\DeclareMathOperator{\Spec}{Spec}
\DeclareMathOperator{\Spf}{Spf}
\DeclareMathOperator{\End}{End}
\DeclareMathOperator{\Hom}{Hom}
\DeclareMathOperator{\height}{ht}
\DeclareMathOperator{\length}{lg}
\DeclareMathOperator{\rk}{rk}

\DeclareMathOperator{\Fitt}{Fitt}

\DeclareMathOperator{\Newt}{Newt}
\DeclareMathOperator{\Hdg}{Hdg}
\DeclareMathOperator{\PR}{PR}
\DeclareMathOperator{\HN}{HN}

\title{The integral Hodge polygon for $p$-divisible groups \\ with endomorphism structure}
\author{Stéphane Bijakowski, Andrea Marrama}
\date{}

\begin{document}

\maketitle

\paragraph{Abstract.}
Let $p$ be a prime number,
let $\O_F$ be the ring of integers of a finite field extension~$F$ of $\Q_p$ and
let $\O_K$ be a complete valuation ring of rank $1$ and mixed characteristic~$(0,p)$.
We introduce and study the \emph{integral Hodge polygon},
a new invariant of $p$-divisible groups~$H$ over $\O_K$ endowed with an action~$\iota$ of $\O_F$.
If $F|\Q_p$ is unramified, this invariant recovers the classical Hodge polygon and
only depends on the reduction of $(H,\iota)$ to the residue field of $\O_K$.
This is not the case in general, whence the attribute~``integral''.
The new polygon lies between
Fargues' Harder-Narasimhan polygons of the $p$-power torsion parts of $H$ and
another combinatorial invariant of $(H,\iota)$ called the Pappas-Rapoport polygon.
Furthermore,
the integral Hodge polygon behaves continuously in families over a $p$-adic analytic space.

\tableofcontents

\section{Introduction}

Let $p$ be a prime number.
When one studies the geometry of the modular curve over a base ring of mixed characteristic~$(0,p)$,
one often considers not the universal elliptic curve~$E$,
but rather its $p$-divisible group~$E[p^\infty]$.
Indeed, Serre-Tate theory ensures that
deforming an elliptic curve in positive characteristic amounts to
deforming the associated $p$-divisible group.
More generally, when studying the $p$-adic geometry of Shimura varieties
(such as the Hilbert modular variety or the Siegel modular variety),
one is especially interested in the $p$-divisible group associated to the universal abelian scheme.

\smallskip\noindent
As far as the special fibre is concerned, several invariants have been attached to
a $p$-divisible group~$H$ over an algebraically closed field of characteristic $p$.
The \emph{Newton polygon}~$\Newt(H)$ classifies the isogeny class of $H$ and
leads to the Newton stratification of Shimura varieties.
The \emph{Hodge polygon}~$\Hdg(H)$ is determined by the dimension of $H$ and
always lies above the Newton polygon (all polygons are concave in this article).

\smallskip\noindent
In the context of Shimura varieties, however,
one is led to consider objects endowed with additional structure,
such as an \emph{endomorphism structure}.
Let $\O_F$ be the ring of integers of a finite field extension~$F$ of $\Q_p$ (possibly ramified) and
assume that $H$ comes with an action~$\iota$ of $\O_F$.
One can then refine the previous invariants and define polygons~$\Newt(H,\iota)$ and $\Hdg(H,\iota)$,
taking this action into account (see \cite[\S 1]{BH1}).
The polygon~$\Newt(H,\iota)$ is just a renormalisation of $\Newt(H)$,
but $\Hdg(H,\iota)$ is genuinely different from $\Hdg(H)$.

\smallskip\noindent
When $F|\Q_p$ is ramified,
the necessity of a good integral model of the Shimura variety suggests to
impose a \emph{Pappas-Rapoport condition} on $(H,\iota)$,
encoding the action of $\O_F$ on the module~$\omega_H$ of invariant differential forms of $H$
(see \cite[Définition~2.2.1]{BH1}).
This condition is based on a fixed combinatorial datum~$\mi$, which in turn determines a new polygon,
the \emph{Pappas-Rapoport polygon}, lying above $\Newt(H,\iota)$.
It is then proved in \cite[Théorème~1.3.1]{BH1} that
$\Hdg(H,\iota)$ lies between $\Newt(H,\iota)$ and the Pappas-Rapoport polygon associated to $\mi$;
its variation throughout the reduction of Shimura varieties is used to
study the geometry of the latter (cf.\ \cite{BH2}).

\smallskip\noindent
Turning now our attention to the generic fibre (after $p$-adic completion),
the objects of interest become $p$-divisible groups with endomorphism structure~$(H,\iota)$ as above,
but defined over a complete valuation ring~$\O_K$ of rank 1 and mixed characteristic $(0,p)$.
In this case, the datum $\mi$ is determined by $(H,\iota)$ itself and
the reduction~$(H_\k,\iota)$ of $(H,\iota)$ to the residue field~$\k$ of $\O_K$ naturally inherits
a compatible Pappas-Rapoport condition, see \S\ref{S-Hdg-PR}.
The associated Pappas-Rapoport polygon has already been considered in \cite{AM},
where it is referred to as the ``Hodge polygon'' of $(H,\iota)$.
Actually, if $F|\Q_p$ is ramified,
this polygon may be different from the Hodge polygon of $(H_\k,\iota)$ mentioned above.
In order to avoid any confusion and in accordance with the previous discussion,
here we write $\PR(H,\iota)$ for the Pappas-Rapoport polygon associated to
the datum~$\mi$ determined by $(H,\iota)$ and call this the ``Pappas-Rapoport polygon'' of $(H,\iota)$.

\smallskip\noindent
Other invariants of $(H,\iota)$ to be considered stem from
the Harder-Narasimhan theory for finite locally free group schemes developed by Fargues in \cite{Fa1}.
Namely, the theory allows one to attach a \emph{Harder-Narasimhan polygon}~$\HN(H[p^i],\iota)$ to
each $p^i$-torsion part of $(H,\iota)$.
These polygons contain information about certain finite locally free sub-group-schemes of $H$ and
converge from above to an invariant~$\HN(H,\iota)$ of the whole $p$-divisible group
(the presence of $\iota$ only accounts for a renormalisation of the polygons in this case);
an application to the $p$-adic geometry of Shimura varieties may be found in \cite{Fa3}.
Under the assumption that $\O_K$ is discretely valued,
it is proved in \cite[Proposition~3.14]{AM} that $\HN(H[p],\iota)$,
and along with it the other Harder-Narasimhan polygons~$\HN(H[p^i],\iota)$,
lie below $\PR(H,\iota)$.

\bigskip\noindent
In this article, we define the \emph{integral Hodge polygon}~$\Hdg^\int(H,\iota)$ of
a $p$-divisible group with endomorphism structure $(H,\iota)$ over $\O_K$,
a new invariant describing the action of a uniformiser $\pi$ of $\O_F$ on $\omega_H$.
If $F|\Q_p$ is unramified,
this polygon only depends on the reduction of $(H,\iota)$ to $\k$ and recovers $\Hdg(H_\k,\iota)$.
This is not the case in general, whence the attribute ``integral''.
The basic feature of the integral Hodge polygon is that
it lies between $\HN(H[p],\iota)$ and $\PR(H,\iota)$.
\begin{thm}[Corollary~\ref{Hdgi-HN}, Theorem~\ref{Hdgi-PR}]
Let $(H,\iota)$ be a $p$-divisible group over $\O_K$ with endomorphism structure for $\O_F$.
Then:
\[
 \HN(H[p],\iota)\le\Hdg^\int(H,\iota)\le\PR(H,\iota).
\]
\end{thm}

\noindent
This refines and generalises the inequality~$\HN(H[p],\iota)\le\PR(H,\iota)$ obtained in
\cite[Proposition~3.14]{AM} for $\O_K$ discretely valued.
More conceptually, this result tells that
the presence of additional (ramified) endomorphism structure on $H$ produces
a constraint on $\HN(H[p],\iota)$,
with consequences on the possible subobjects of $H$, see Remark~\ref{Hdgi-HN-rmk}.

\smallskip\noindent
The integral Hodge polygon of $(H,\iota)$ is in general unrelated to
the Newton polygon and the Hodge polygon of $(H_\k,\iota)$, see \S\ref{S-ex}.
An exception is the limit case when $\Hdg^\int(H,\iota)=\PR(H,\iota)$,
which happens if and only if $\Hdg(H_\k,\iota)=\PR(H,\iota)$, see Proposition~\ref{max}.
This situation is realised for instance when $\O_F$ acts on $\omega_H$ through
a fixed embedding~$\O_F\to\O_K$, that is, when $(H,\iota)$ is a $p$-divisible \emph{$\O_F$-module}.
The case of $\O_F$-modules also falls within a class of
objects which we call \emph{$\pi$-diagonalisable},
meaning that the action of $\pi$ on $\omega_H$ is diagonalisable.
This condition is detected by $\Hdg^\int(H,\iota)$ if $F|\Q_p$ is tamely ramified and
it always implies that $\Hdg^\int(H,\iota)=\Hdg(H_\k,\iota)$, see \S\ref{S-pid}.

\smallskip\noindent
Another notable feature of the integral Hodge polygon,
especially in view of geometric applications, is that it behaves continuously in families.
\begin{thm}[Theorem~\ref{cont}]
 Let $\X$ be a formal scheme as in \S\ref{S-cont} and
 denote by $\X^\an$ its generic fibre as a Berkovich analytic space.
 Let $(\H,\iota)$ be a $p$-divisible group over $\X$ with endomorphism structure for $\O_F$ and
 suppose that $\H$ has constant height.
 Then, the function~$\Hdg^\int(\H,\iota)$ from $\X^\an$ to the space of polygons is continuous.
 Moreover, for every fixed polygon $f_0$, the subset:
 \[
  \X^\an_{\Hdg^\int\le f_0}\coloneqq\Set{x\in\X^\an|\Hdg^\int(\H_x,\iota)\le f_0}\subseteq\X^\an
 \]
 defines a closed analytic domain of $\X^\an$.
\end{thm}

\paragraph{Acknowledgements.}
The first author is part of the project ANR-19-CE40-0015 COLOSS.
The second author was funded by the Fondation Mathématique Jacques Hadamard,
while staying at the Centre de Mathématiques Laurent Schwartz, École Polytechnique.
The authors would like to thank Lorenzo Fantini, Laurent Fargues and
Valentin Hernandez for helpful discussions.
They also thank the anonymous referee for some comments that improved the exposition.

\section{Setup and notation}

Let $p$ be a prime number.
Let $K$ be a field extension of the $p$-adic numbers~$\Q_p$ and
assume that $K$ is complete with respect to
a valuation~$v$ with values in $\R$ extending the $p$-adic valuation (thus normalised at $v(p)=1$).
Denote by $\O_K$ the valuation ring of $K$ and by $\k$ its residue field,
which is then of characteristic $p$.

Let $F$ be a finite field extension of $\Q_p$ of degree $d$,
with ring of integers $\O_F$ and residue field $\k_F$;
we choose a uniformiser $\pi\in\O_F$.
Denote by $F^\nr$ the maximal unramified subextension (or \emph{inertia subfield})
of $F|\Q_p$ and by $\O_{F^\nr}$ its ring of integers,
so that $F^\nr|\Q_p$ is an unramified extension of degree $f(F|\Q_p)$, the inertia degree of $F|\Q_p$,
and $F|F^\nr$ is a totally ramified extension of degree $e(F|\Q_p)$,
the ramification index of $F|\Q_p$.

We assume throughout the document that $\k$ is perfect and that $K$ contains a Galois closure of
$F$ over $\Q_p$ (and hence that $\k$ contains $\k_F$), although we do not fix an embedding.
However, let us remark that the main definitions can be given without these assumptions,
as they are invariant under suitable base change (we will indicate when this is the case).
In particular, the statements depending only on these definitions hold in general.

Write $W(\k)$ for the ring of Witt vectors with coefficients in $\k$,
which is naturally a subring of $\O_K$,
and let $K_0$ be its fraction field, a subfield of $K$.
We denote by $\sigma$ the Frobenius endomorphism of $W(\k)$ and its extension to $K_0$.

For $n\in\N$, we write $\R^n_+\coloneqq\Set{(a_i)_{i=1}^n\in\R^n|a_1\ge\dots\ge a_n}$ for
the set of decreasing $n$-tuples of real numbers.
As a subset of $\R^n$,
note that $\R^n_+$ is closed under addition and nonnegative scalar multiplication.
We endow $\R^n_+$ with the following partial order:
\[
 (a_i)_{i=1}^n\le (b_i)_{i=1}^n \qquad\text{if}\qquad
 \sum_{i=1}^j a_i\le\sum_{i=1}^j b_i \quad\text{for all $1\le j\le n$ and}\quad
 \sum_{i=1}^n a_i=\sum_{i=1}^n b_i.
\]
An element $f=(a_i)_{i=1}^n\in\R^n_+$ can be viewed as a piecewise affine linear, continuous,
concave function $f\colon[0,n]\to\R$ starting at $(0,0)$ and
proceeding with slope $a_i$ on $[i-1,i]$.
In this sense, there is an obvious notion of \emph{break points} of $f$,
from which we exclude the extremal points $(0,0)$ and $(n,f(n))$.
The partial order defined above extends naturally to the set of all piecewise affine linear,
continuous, concave functions $f\colon[0,n]\to\R$ with $f(0)=0$,
namely $f\le g$ if we have pointwise inequality and $f(n)=g(n)$.

\section{Review and definitions}

\subsection{\texorpdfstring{$p$}{p}-divisible groups with endomorphism structure}

Let $R$ be a commutative $p$-adically complete local ring.

We call \emph{$p$-group} over $R$ any finite locally free commutative group scheme over $\Spec R$ of
$p$-power order.
If $H$ is such an object, we denote by $\height H$ the height of $H$,
i.e.\ the logarithm to base $p$ of its order,
and we write $\omega_H$ for the cotangent space of $H$ along the identity section,
a finitely presented $R$-module.
Recall that the association $H\mapsto\omega_H$, from the category of
finite locally free commutative group schemes over $\Spec R$ to that of $R$-modules,
defines a contravariant additive functor which is compatible with base change and right exact
(see \cite[Proposition~II.3.3.4]{Me}).
The notation $H^D$ stands for the Cartier dual of $H$.

For $H=(H[p^i])_{i\ge1}$ a $p$-divisible group over $\Spec R$ (or, for short, over $R$),
we denote by $\height H$ the height of $H$ (which equals the height of the $p$-group $H[p]$ over $R$)
and we set $\omega_H\coloneqq\varprojlim_{i\ge1}\omega_{H[p^i]}$.
If $p$ is nilpotent in $R$, then $\omega_H=\omega_{H[p^i]}$ for $i\ge1$ sufficiently large and
this is a finite free $R$-module (see \cite[\S 3.3.1]{BBM} and recall that $R$ is a local ring).
In general, since $R$ is $p$-adically complete,
we have that $\omega_H$ is anyway a finite free $R$-module,
with $\omega_H/p^i\omega_H=\omega_{H[p^i]}$ for all $i\ge1$.
The dimension of $H$, denoted by $\dim H$, is the rank of $\omega_H$ over $R$.
The association $H\mapsto\omega_H$,
from the category of $p$-divisible groups over $R$ to that of $R$-modules,
defines a contravariant $\Z_p$-linear functor which is compatible with base change
(to other $p$-adically complete local rings).
We write $H^D=(H[p^i]^D)_{i\ge1}$ for the Cartier dual of $H$.

If $H$ is a $p$-divisible group over $\k$,
we denote by $(\D(H),\phi_H)$ its contravariant Dieudonné module (see \cite[\S III]{Fo1}).
Recall that this is composed of a free $W(\k)$-module $\D(H)$ of rank $\height H$ and
an injective $\sigma$-linear endomorphism $\phi_H\colon\D(H)\to\D(H)$ such that
$p\D(H)\subseteq\phi_H\D(H)$;
in particular, $(\D(H),\phi_H)$ is an $F$-crystal over $\k$ as in \cite[Définition~1.1.1]{BH1}.
The association $H\mapsto(\D(H),\phi_H)$ determines a contravariant $\Z_p$-linear functor from
the category of $p$-divisible groups over $\k$ to that of $F$-crystals over $\k$,
inducing an antiequivalence with the full subcategory of Dieudonné modules.
This functor is compatible with base change of perfect fields.
Moreover, we have a natural identification of $\k$-vector-spaces:
\begin{equation}\label{Die-ctg}
 \D(H)/\phi_H\D(H)\cong\omega_H.
\end{equation}

\begin{dfn}
 A \emph{$p$-group}, respectively \emph{$p$-divisible group} over $R$
 \emph{with endomorphism structure} for $\O_F$ is a pair~$(H,\iota)$ consisting of
 a $p$-group, respectively a $p$-divisible group~$H$ over $\Spec R$ and
 a map of $\Z_p$-algebras $\iota\colon\O_F\to\End(H)$.
\end{dfn}

We denote by $\pgr_{R,\O_F}$, respectively $\pdiv_{R,\O_F}$ the category of
$p$-groups, respectively $p$-divisible groups over $R$ with endomorphism structure for $\O_F$,
with morphisms given by maps of $p$-groups, respectively $p$-divisible groups over $R$ that are
compatible with $\iota$ (or $\O_F$-\emph{equivariant}).

\begin{rmk}
Note that the definition of the Dieudonné module used here is the one given by
Fontaine in \cite[\S III]{Fo1}.
One may also use the Dieudonné crystal defined by
Berthelot, Breen and Messing in \cite[\S 3.3]{BBM}, evaluated at $W(\k)$.
The latter is naturally isomorphic to the former up to a Frobenius twist (see \cite[\S 4.2]{BBM}).
\end{rmk}

\subsection{The Hodge polygon in special fibre}

Let $(H,\iota)$ be a $p$-divisible group over $\k$ with endomorphism structure for $\O_F$,
where we remind that $\k$ is a perfect field of characteristic $p$ containing
the residue field $\k_F$ of $F$.
Let us recall from \cite[\S 1.1]{BH1} the definition of the Hodge polygon of $(H,\iota)$,
which is based on a more general invariant of $F$-crystals with $\O_F$-action.

In fact,
the Dieudonné module $(\D(H),\phi_H)$ of $H$ inherits from $\iota$ a $\Z_p$-linear action of $\O_F$,
that is, a map of $\Z_p$-algebras $\iota\colon\O_F\to\End(\D(H),\phi_H)$.
Because $\k$ contains $\k_F$, we have a decomposition of finite free $W(\k)$-modules:
\[
 \D(H)=\bigoplus_{\upsilon\colon F^\nr\to K_0}\D(H)_\upsilon,
\]
where $\upsilon$ ranges through the $f(F|\Q_p)$ embeddings of $F^\nr$ in $K_0$ and
$\O_{F^\nr}$ acts on $\D(H)_\upsilon$ via $\upsilon\colon\O_{F^\nr}\to W(\k)$.
The $\sigma$-linear endomorphism $\phi_H$ restricts then to injective maps:
\[
 \phi_H\colon\D(H)_{\sigma^{-1}\upsilon}\longrightarrow\D(H)_\upsilon,
\]
ensuring that the ranks $\rk_{W(\k)}\D(H)_\upsilon$ are all the same.
Moreover, the decomposition above reduces to a decomposition of $\k$-vector-spaces:
\begin{equation}\label{ctg-dcp-unr-spf}
 \D(H)/\phi_H\/D(H)=\bigoplus_\upsilon\D(H)_\upsilon/\phi_H\D(H)_{\sigma^{-1}\upsilon}.
\end{equation}
Observe now, for each embedding~$\upsilon$, that $\D(H)_\upsilon$ is a module over the ring:
\[
 W_{\O_F,\upsilon}(\k)\coloneqq\O_F\otimes_{\O_{F^\nr},\upsilon}W(\k)
\]
and that this is a discrete valuation ring, with uniformiser $\pi\otimes1$;
in fact, it is the ring of integers of a totally ramified extension of $K_0$ of degree $e(F|\Q_p)$.
Without ambiguity, we denote by $v$ the valuation of $W_{\O_F,\upsilon}(\k)$ normalised at $v(p)=1$.
Because $\D(H)_\upsilon$ is a torsion free $W(\k)$-module,
it is again torsion free and hence free as a $W_{\O_F,\upsilon}(\k)$-module.
In fact, we have that $\rk_{W(\k)}\D(H)_\upsilon=e(F|\Q_p)\rk_{W_{\O_F,\upsilon}(\k)}\D(H)_\upsilon$
independently of $\upsilon$, so we may set:
\[
 n\coloneqq\rk_{W_{\O_F,\upsilon}(\k)}\D(H)_\upsilon\in\N
\]
for any $\upsilon$.
We remark that this means in particular that:
\[
 \height H=\rk_{W(\k)}\D(H)=\sum_\upsilon\rk_{W(\k)}\D(H)_\upsilon=
 f(F|\Q_p)e(F|\Q_p)n=dn\in d\N.
\]
Now, for every $\upsilon\colon F^\nr\to K_0$,
since $W_{\O_F,\upsilon}(\k)$ is a valuation ring and $\phi_H$ is injective, we may write:
\[
 \bar{\D}_\upsilon\coloneqq\D(H)_\upsilon/\phi_H\D(H)_{\sigma^{-1}\upsilon}\cong
 \bigoplus_{i=1}^n W_{\O_F,\upsilon}(\k)/a_{\upsilon,i}W_{\O_F,\upsilon}(\k)
\]
for some nonzero elements~$a_{\upsilon,1},\dots,a_{\upsilon,n}\in W_{\O_F,\upsilon}(\k)$ with
$v(a_{\upsilon,1})\ge\dots\ge v(a_{\upsilon,n})$;
these elements are uniquely determined up to units of $W_{\O_F,\upsilon}(\k)$.
We set:
\[
 \Hdg(H,\iota)_\upsilon\coloneqq(v(a_{\upsilon,1}),\dots,v(a_{\upsilon,n}))\in\R^n_+.
\]
As a piecewise affine linear function,
$\Hdg(H,\iota)_\upsilon\colon[0,n]\to\R$ interpolates the following values:
\[
 \Hdg(H,\iota)_\upsilon(i)=\sum_{j=1}^i v(a_{\upsilon,j})
 =v(\Fitt_0\bar{\D}_\upsilon)-v(\Fitt_i\bar{\D}_\upsilon),
 \quad i\in\Set{0,\dots,n},
\]
where $\Fitt_i\bar{\D}_\upsilon$ denotes the $i$-th Fitting ideal of
the $W_{\O_F,\upsilon}(\k)$-module~$\bar{\D}_\upsilon$;
since $\Fitt_i\bar{\D}_\upsilon$ is a principal ideal, it makes sense to consider its valuation.
Note that $\bar{\D}_\upsilon$ is a $p$-torsion $W_{\O_F,\upsilon}(\k)$-module.
Thus, if we write:
\[
 \bar{\D}_\upsilon[\pi^j]\coloneqq\Set{w\in\bar{\D}_\upsilon|(\pi\otimes1)^jw=0}
  \subseteq\bar{\D}_\upsilon
\]
for $0\le j\le e(F|\Q_p)$, then we have the following alternative description:
\[
 \Hdg(H,\iota)_\upsilon\colon x\longmapsto\frac{1}{e(F|\Q_p)}
 \sum_{j=1}^{e(F|\Q_p)}\min\Set{x,\dim_\k\bar{\D}_\upsilon[\pi^j]/\bar{\D}_\upsilon[\pi^{j-1}]},
 \quad x\in[0,n].
\]
Indeed, $\dim_\k\bar{\D}_\upsilon[\pi^j]/\bar{\D}_\upsilon[\pi^{j-1}]$
equals the number of indices $i\in\Set{1,\dots,n}$ such that $e(F|\Q_p)v(a_{\upsilon,i})\ge j$.
In particular, the end point of $\Hdg(H,\iota)_\upsilon$ is:
\begin{equation}\label{Hdg-ups-end}
 \Hdg(H,\iota)_\upsilon(n)=\frac{1}{e(F|\Q_p)}
 \sum_{j=1}^{e(F|\Q_p)}\dim_\k\bar{\D}_\upsilon[\pi^j]/\bar{\D}_\upsilon[\pi^{j-1}]
 =\frac{1}{e(F|\Q_p)}\dim_\k\bar{\D}_\upsilon.
\end{equation}
The \emph{Hodge polygon} of $(H,\iota)$ is defined to be:
\[
 \Hdg(H,\iota)\coloneqq\frac{1}{f(F|\Q_p)}\sum_\upsilon\Hdg(H,\iota)_\upsilon\in\R^n_+,
\]
where $\upsilon$ ranges through the $f(F|\Q_p)$ embeddings of $F^\nr$ in $K_0$.
Its end point is given by the average over $\upsilon$ of the equations~\eqref{Hdg-ups-end}:
\[
 \Hdg(H,\iota)(n)=\frac{1}{f(F|\Q_p)}\sum_\upsilon\frac{1}{e(F|\Q_p)}\dim_\k\bar{\D}_\upsilon=
  \frac{1}{d}\dim_\k\D(H)/\phi_H\/D(H)=\frac{\dim H}{d},
\]
the last equality due to the identification~\eqref{Die-ctg}.
We remark that up to reversing the order of the slopes in order to get a convex polygon,
$\Hdg(H,\iota)$ equals the Hodge polygon of the $F$-crystal~$(\D(H),\phi_H)$ with
$\O_F$-action~$\iota$, as introduced in \cite[Définition~1.1.7]{BH1}.

Note that the above definition is invariant under base change of $(H,\iota)$ to another perfect field.
Thus, one may define $\Hdg(H,\iota)$ even if $\k$ is not perfect or does not contain $\k_F$,
namely as the Hodge polygon of the base change of $(H,\iota)$ to
a sufficiently large perfect field extension of $\k$.
In this case, we still find that $\height H=dn$ for some $n\in\N$ and so $\Hdg(H,\iota)\in\R^n_+$.
All statements concerning $\Hdg(H,\iota)$ will then hold in this more general setup.

\begin{rmk}\label{Hdg-unr}
 If $F|\Q_p$ is an unramified extension,
 then we have that $W_{\O_F,\upsilon}(\k)=W(\k)$ for every embedding $\upsilon$ of $F^\nr=F$ in $K_0$.
 In particular, for $(H,\iota)\in\pdiv_{\k,\O_F}$ one sees that in this case:
 \[
  \Hdg(H,\iota)_\upsilon\coloneqq(1,\dots,1,0,\dots,0),
 \]
 the number of $1$'s being equal to $\dim_\k\D(H)_\upsilon/\phi_H\D(H)_{\sigma^{-1}\upsilon}$.
 In fact, the notion of Hodge polygon in this setting can be traced back to that of
 ``Hodge point'' for $F$-crystals with additional (unramified) structure,
 introduced in \cite[\S 4]{RR}.
\end{rmk}

\subsection{The Pappas-Rapoport polygon}\label{S-PR}

Let $(H,\iota)$ be a $p$-divisible group over $\O_K$ with endomorphism structure for $\O_F$,
where we remind that $\O_K$ is the valuation ring of a complete valued field extension~$(K,v)$ of
$\Q_p$ containing a Galois closure of $F$ over $\Q_p$.
We will now define the Pappas-Rapoport polygon of $(H,\iota)$.
Note that $\height H=\height H_\k\in d\N$,
where $H_\k$ denotes the reduction of $H$ to the residue field~$\k$ of $\O_K$;
thus, we may write $\height H=dn$ with $n\in\N$.
Set $\omega_{H,K}\coloneqq\omega_H\otimes_{\O_K}K$ and
observe that $\iota$ induces a map of $\Q_p$-algebras $\iota\colon F\to\End_K(\omega_{H,K})$.
Because $K$ contains a Galois closure of $F$ over $\Q_p$, we have a decomposition of $K$-vector-spaces:
\begin{equation}\label{ctg-dcp}
 \omega_{H,K}=\bigoplus_{\tau\colon F\to K}\omega_{H,K,\tau},
\end{equation}
given by $\omega_{H,K,\tau}=\Set{w\in \omega_{H,K}|\forall a\in F:\iota(a)(w)=\tau(a)w}$ for
each embedding~$\tau$ of $F$ in $K$.
Set $r_\tau\coloneqq\dim_K\omega_{H,K,\tau}$ for every $\tau$.

\begin{rmk}\label{ctg-dcp-dim}
 Recall from \cite[Lemma~1.13]{Bi} that we have an $\O_F$-equivariant exact sequence of $\O_K$-modules:
 \begin{equation}\label{Hdg-fil}
  0\longrightarrow\omega_H\longrightarrow\E\longrightarrow\omega_{H^D}^\vee\longrightarrow 0,
 \end{equation}
 where $\E$ is a free $\O_F\otimes_{\Z_p}\O_K$-module of rank $n$ and
 $\omega_{H^D}^\vee\coloneqq\Hom_{\O_K}(\omega_{H^D},\O_K)$ carries
 an $\O_F$-action induced naturally from $\iota$.
 Write $\E_K\coloneqq\E\otimes_{\O_K}K=\bigoplus_\tau\E_{K,\tau}$ in
 a similar fashion as for $\omega_{H,K}$ above.
 Since $\E$ is free of rank $n$ over $\O_F\otimes_{\Z_p}\O_K$,
 we have that $\dim_K\E_{K,\tau}=n$ for every $\tau\colon F\to K$.
 Hence, because the $K$-vector-space~$\omega_{H,K,\tau}$ injects into $\E_{K,\tau}$,
 we find that $r_\tau=\dim_K\omega_{H,K,\tau}\le n$ for all $\tau$'s.
\end{rmk}

We may now set:
\[
 \PR(H,\iota)_\tau\coloneqq(1,\dots,1,0,\dots,0)\in\R^n_+,
\]
the number of $1$'s being equal to $r_\tau$, for each $\tau\colon F\to K$.
The \emph{Pappas-Rapoport polygon} of $(H,\iota)$ is defined to be:
\[
 \PR(H,\iota)\coloneqq\frac{1}{d}\sum_\tau\PR(H,\iota)_\tau\in\R^n_+,
\]
where $\tau$ ranges through the $d$ embeddings of $F$ in $K$.
Equivalently, as a piecewise affine linear function:
\[
 \PR(H,\iota)\colon x\mapsto\frac{1}{d}\sum_\tau\min\Set{x,r_\tau},\quad x\in[0,n].
\]
Its end point is given by:
\[
 \PR(H,\iota)(n)=\frac{1}{d}\sum_\tau r_\tau=\frac{1}{d}\dim_K\omega_{H,K}=\frac{\dim H}{d}.
\]

Note that the above definition is invariant under base change of $(H,\iota)$ to
the valuation ring of a larger field $K$ within our setup.
Thus, one may define $\PR(H,\iota)$ even if $K$ does not contain a Galois closure of $F$ over $\Q_p$, namely as the Pappas-Rapoport polygon of the base change of $(H,\iota)$ to
the valuation ring of a sufficiently large finite field extension of $K$.
In addition, we did not really make use of the fact that $\k$ is perfect,
so the Pappas-Rapoport polygon is also defined without this assumption.
All statements concerning $\PR(H,\iota)$ will then hold in this more general setup.

\begin{rmk}\label{rmk-PR-nm}
 In the case that the valuation of $K$ is discrete, the Pappas-Rapoport polygon of
 an object~$(H,\iota)\in\pdiv_{\O_K,\O_F}$ was already considered in \cite{AM},
 although in loc.\ cit.\ it is named ``Hodge polygon'' of $(H,\iota)$.
 Here, in accordance with the terminology of \cite{BH1},
 we choose to reserve the latter name for an invariant of objects over $\k$.
 The name ``Pappas-Rapoport polygon'' comes then from \cite{BH1} as well and
 its use here is motivated by the following observation.

 For each embedding~$\upsilon$ of $F^\nr$ in $K_0$,
 let $\tau_{\upsilon,1},\dots,\tau_{\upsilon,e(F|\Q_p)}$ be an ordering of
 the embeddings~$\tau\colon F\to K$ which restrict to $\upsilon$ on $F^\nr$.
 Set then $r_{\upsilon,i}\coloneqq\dim_K\omega_{H,K,\tau_{\upsilon,i}}$,
 with notation as in \eqref{ctg-dcp}, for every $\upsilon$ and $1\le i\le e(F|\Q_p)$.
 Now, up to reversing the order of the slopes in order to get a convex polygon,
 one finds that $\PR(H,\iota)$ corresponds to the ``Pappas-Rapoport polygon'' associated to
 the tuple~$\mi\coloneqq(r_{\upsilon,i})_{\upsilon,1\le i\le e(F|\Q_p)}$ in \cite[\S 1.2]{BH1}.
\end{rmk}

\begin{rmk}\label{PR-unr}
 If $F|\Q_p$ is unramified,
 the decomposition~\eqref{ctg-dcp} for $(H,\iota)\in\pdiv_{\O_K,\O_F}$ restricts to
 a decomposition of $\omega_H$ into a direct sum of finite free $\O_K$-submodules,
 on each of which $\O_F$ acts via a single embedding of $F=F^\nr$ in $K_0\subseteq K$
 (see \eqref{ctg-dcp-unr} below).
 The reduction of this to $\k$ recovers, via \eqref{Die-ctg},
 the decomposition~\eqref{ctg-dcp-unr-spf} relative to
 the reduction $(H_\k,\iota)$ of $(H,\iota)$ to $\k$.
 Taking Remark~\ref{Hdg-unr} into consideration,
 we get that in this case $\PR(H,\iota)=\Hdg(H_\k,\iota)$.

 For a general extension $F|\Q_p$, we will see in the next section that
 one can still compare the two polygons~$\PR(H,\iota)$ and $\Hdg(H_\k,\iota)$,
 see Proposition~\ref{Hdg-PR}.
\end{rmk}

\subsection{Comparison between Pappas-Rapoport and Hodge polygons}\label{S-Hdg-PR}

Let us keep the notation of the previous section,
with $(H,\iota)\in\pdiv_{\O_K,\O_F}$ and $\height H=nd$.
Consider the $\O_F$-action induced by $\iota$ on $\omega_H$,
a map of $\Z_p$-algebras $\iota\colon\O_F\to\End_{\O_K}(\omega_H)$.
Looking at the restriction of $\iota$ to $\O_{F^\nr}$,
we have a decomposition of finite free $\O_K$-modules:
\begin{equation}\label{ctg-dcp-unr}
 \omega_H=\bigoplus_{\upsilon\colon F^\nr\to K_0}\omega_{H,\upsilon},
\end{equation}
where $\upsilon$ ranges through the $f(F|\Q_p)$ embeddings of $F^\nr$ in $K_0$ and $\O_{F^\nr}$ acts on
$\omega_{H,\upsilon}$ via $\upsilon\colon\O_{F^\nr}\to W(\k)\subseteq\O_K$.
Observe now that the $\O_F$-action~$\iota$ restricts to each component~$\omega_{H,\upsilon}$ and
the decomposition~\eqref{ctg-dcp} restricts to:
\begin{equation}\label{ctg-dcp-ups}
 \omega_{H,\upsilon}\otimes_{\O_K}K=\bigoplus_{\tau|\upsilon}\omega_{H,K,\tau},
\end{equation}
where $\tau|\upsilon$ stands for the embeddings of $F$ in $K$ which agree with $\upsilon$ on $F^\nr$.
Fix then an ordering~$\tau_{\upsilon,1},\dots,\tau_{\upsilon,e(F|\Q_p)}$ of
the set~$\Set{\tau\colon F\to K|\tau|\upsilon}$ for every $\upsilon\colon F^\nr\to K_0$ and
consider the filtrations:
\begin{equation}\label{PR-fil}
 0=\omega_{\upsilon,0}\subseteq\omega_{\upsilon,1}\subseteq\dots
 \subseteq\omega_{\upsilon,e(F|\Q_p)}=\omega_{H,\upsilon}
\end{equation}
given by the $\O_K$-submodules:
\[
 \omega_{\upsilon,i}\coloneqq\omega_{H,\upsilon}\cap\bigoplus_{j=1}^{i}\omega_{H,K,\tau_{\upsilon,j}}
\]
for $0\le i\le e(F|\Q_p)$.
In order to make use of these filtrations, we will need the first part of the following lemma
(the final statement will only be useful later on).

\begin{lem}\label{Hdg-PR-lem}
 Let $M$ be a finitely generated and free $\O_K$-module,
 together with a map of $\Z_p$-algebras $\iota\colon\O_F\to\End_{\O_K}(M)$.
 Consider the decomposition of $K$-vector-spaces:
 \[
  M_K\coloneqq M\otimes_{\O_K}K=\bigoplus_{\tau\colon F\to K}M_{K,\tau}
 \]
 given by $M_{K,\tau}=\Set{w\in M_K|\forall a\in F:\iota(a)(w)=\tau(a)w}$ for
 each embedding~$\tau$ of $F$ in $K$,
 where $\iota$ is extended in the natural way to an $F$-action on $M_K$.
 Then, for any subset~$I\subseteq\Set{\tau\colon F\to K}$ the $\O_K$-module:
 \[
  M_I\coloneqq M\cap\bigoplus_{\tau\in I}M_{K,\tau}\subseteq M_K
 \]
 is a direct summand of $M$, free of rank $\sum_{\tau\in I}\dim_K M_{K,\tau}$.
 Moreover, setting:
 \[
  \rho_I\coloneqq\prod_{\tau\in I}\tau(\pi)\in\O_K,
 \]
 we have that $\rho_I M_I\subseteq\iota(\pi)M$.
\end{lem}

\begin{proof}
 The inclusion~$M\subseteq M_K$ induces an injective map:
 \begin{equation}\label{lem-M_I-eq1}
  M/M_I\longrightarrow M_K/\bigoplus_{\tau\in I}M_{K,\tau}\cong\bigoplus_{\tau\notin I}M_{K,\tau},
 \end{equation}
 ensuring that $M/M_I$ is a torsion free $\O_K$-module.
 Since $\O_K$ is a valuation ring and $M/M_I$ is finitely generated over it,
 it follows that $M/M_I$ is a free $\O_K$-module and, in particular, projective.
 This proves that $M_I$ is a direct summand of $M$ and, as such,
 it is also torsion free and finitely generated over $\O_K$, hence a free $\O_K$-module.
 Now, the inclusions $M_I\subseteq\bigoplus_{\tau\in I}M_{K,\tau}$ and
 \eqref{lem-M_I-eq1} imply respectively that:
 \[
  \rk_{\O_K}M_I\le\sum_{\tau\in I}\dim_K M_{K,\tau}
  \qquad\text{and}\qquad
  \rk_{\O_K}M/M_I\le\sum_{\tau\notin I}\dim_K M_{K,\tau}.
 \]
 However, the left-hand sides of the above expressions sum to $\rk_{\O_K}M=\dim_K M_K$,
 which also equals the sum of the right-hand sides.
 Thus, we actually have equality in both expressions.

 We prove the final statement by induction on the cardinality of $I$.
 If $I$ is empty, then $M_I=0$ and $\rho_I=1$, so the assertion is clear.
 Assume now that $I$ contains at least one element and fix $\tau_0\in I$.
 Set then $I'\coloneqq I\setminus\Set{\tau_0}$ and
 observe that $M_{I'}=M_I\cap\bigoplus_{\tau\in I'}M_{K,\tau}$.
 In particular, the inclusion $M_I\subseteq\bigoplus_{\tau\in I}M_{K,\tau}$ induces an injective map:
 \[
  M_I/M_{I'}\longrightarrow\bigoplus_{\tau\in I}M_{K,\tau}/\bigoplus_{\tau\in I'}M_{K,\tau}
   \cong M_{K,\tau_0},
 \]
 ensuring that the $\O_F$-action induced by $\iota$ on $M_I/M_{I'}$ factors through
 $\tau_0\colon\O_F\to\O_K$.
 Thus, for every element $w\in M_I$,
 we have that $\tau_0(\pi)w$ and $\iota(\pi)(w)$ have the same image in $M_I/M_{I'}$,
 i.e.\ their difference lies in $M_{I'}$.
 By inductive hypothesis, then:
 \[
  \rho_{I'}(\tau_0(\pi)w-\iota(\pi)(w))\in\iota(\pi)M.
 \]
 But this element can be rewritten as $\rho_I w-\iota(\pi)(\rho_{I'}w)$,
 so that $\rho_I w$ belongs itself to $\iota(\pi)M$,
 proving the inductive step and hence the claimed statement.
\end{proof}

Back to our discussion, the lemma implies that
for every embedding $\upsilon$ of $F^\nr$ in $K_0$ the filtration~\eqref{PR-fil} of
$\omega_{H,\upsilon}$ is by $\O_K$-direct-summands,
with $\rk_{\O_K}\omega_{\upsilon,i}/\omega_{\upsilon,i-1}=\dim_K\omega_{H,K,\tau_{\upsilon,i}}$ for
$1\le i\le e(F|\Q_p)$.
Furthermore, these filtrations are clearly $\O_F$-stable and the $\O_F$-action induced on
each graded piece~$\omega_{\upsilon,i}/\omega_{\upsilon,i-1}$ is via
$\tau_{\upsilon,i}\colon\O_F\to\O_K$.

Let now $(H_\k,\iota)$ denote the reduction of $(H,\iota)$ to $\k$.
Considering the $\O_F$-action induced by $\iota$ on $\omega_{H_\k}$,
or rather its restriction to $\O_{F^\nr}$, we have a decomposition of $\k$-vector-spaces:
\begin{equation}\label{ctg-dcp-unr-red}
 \omega_{H_\k}=\bigoplus_\upsilon\omega_{H_\k,\upsilon},
\end{equation}
where $\upsilon$ ranges through the embeddings of $F^\nr$ in $K_0$ and $\O_{F^\nr}$ acts on
$\omega_{H_\k,\upsilon}$ through $\upsilon\colon\O_{F^\nr}\to W(\k)\to\k$ (in fact,
via \eqref{Die-ctg} this recovers the decomposition~\eqref{ctg-dcp-unr-spf} relative to $H_\k$).
Note that \eqref{ctg-dcp-unr-red} identifies $\O_F$-equivariantly with
the reduction to $\k$ of \eqref{ctg-dcp-unr}.
Then, putting together the previous observations,
we have that for every $\upsilon$ the reduction to $\k$ of \eqref{PR-fil} gives a filtration:
\begin{equation}\label{PR-fil-red}
 0=\omega_{\upsilon,0,\k}\subseteq\omega_{\upsilon,1,\k}\subseteq\dots
 \subseteq\omega_{\upsilon,e(F|\Q_p),\k}=\omega_{H_\k,\upsilon}\cong\omega_{H,\upsilon}\otimes_{\O_K}\k
\end{equation}
of $\omega_{H_\k,\upsilon}$ by sub-$\k$-vector-spaces satisfying the following conditions:
\begin{itemize}
 \item $\iota(\pi)(\omega_{\upsilon,i,\k})\subseteq\omega_{\upsilon,i-1,\k}$ for $1\le i\le e(F|\Q_p)$,
  as the $\O_F$-action induced on the quotient~$\omega_{\upsilon,i,\k}/\omega_{\upsilon,i-1,\k}$ is
  via $\tau_{\upsilon,i}\colon\O_F\to\O_K\to\k$, which sends $\pi$ to zero;
 \item $\dim_\k\omega_{\upsilon,i,\k}/\omega_{\upsilon,i-1,\k}=\dim_K\omega_{H,K,\tau_{\upsilon,i}}$
  for $1\le i\le e(F|\Q_p)$.
\end{itemize}
In other words, according to \cite[Définition~2.2.1]{BH1},
the collection of the filtrations~\eqref{PR-fil-red} for varying $\upsilon$ yields
a \emph{Pappas-Rapoport condition} on $(H_\k,\iota)$ for
the tuple~$\mi=(r_{\upsilon,i})_{\upsilon,i}$ defined by
$r_{\upsilon,i}\coloneqq\dim_K\omega_{H,K,\tau_{\upsilon,i}}$ for
all $\upsilon\colon F^\nr\to K_0$ and $1\le i\le e(F|\Q_p)$.
Equivalently, we have a Pappas-Rapoport condition for $\mi$,
in the sense of \cite[Définition~1.2.1]{BH1},
on the Dieudonné module~$(\D(H_\k),\phi_{H_\k})$ of $H_\k$,
endowed with the $\O_F$-action corresponding to $\iota$.
This allows us to deduce the next proposition from \cite[Théorème~1.3.1]{BH1}
(after reversing the order of the slopes of the polygons and
taking the second observation of Remark~\ref{rmk-PR-nm} into account).

\begin{prp}[\cite{BH1}, Théorème~1.3.1]\label{Hdg-PR}
 Let $(H,\iota)$ be a $p$-divisible group over $\O_K$ with endomorphism structure for $\O_F$ and
 denote by $(H_\k,\iota)$ the reduction of $H$ to $\k$ together with the induced $\O_F$-action.
 Then:
 \[
  \Hdg(H_\k,\iota)\le\PR(H,\iota).
 \]
\end{prp}

\begin{rmk}\label{rmk-gRc}
 A closer look at the proof of \cite[1.3.1]{BH1} shows that in fact,
 for $(H,\iota)\in\pdiv_{\O_K,\O_F}$ and its reduction $(H_\k,\iota)\in\pdiv_{\k,\O_F}$, we have:
 \[
  \Hdg(H_\k,\iota)_\upsilon\le\frac{1}{e(F|\Q_p)}\sum_{\tau|\upsilon}\PR(H,\iota)_\tau
 \]
 for all embeddings~$\upsilon$ of $F^\nr$ in $K_0$.
 Consider now the limit case $\Hdg(H_\k,\iota)=\PR(H,\iota)$.
 In the language of \cite[Définition~3.1.1]{BH1},
 this means that $(H_\k,\iota)$ satisfies the \emph{generalised Rapoport condition} with respect to
 the tuple~$\mi$ introduced above.
 Note that this is equivalent to having an equality in all expressions above.
 Following then a similar observation as in loc.\ cit.,
 one finds that this condition is in turn equivalent to the fact that for every $\upsilon$
 the filtration~\eqref{PR-fil-red} is given by:
 \[
  \omega_{\upsilon,i,\k}=\omega_{H_\k,\upsilon}[\pi^i]
   \coloneqq\Set{w\in\omega_{H_\k,\upsilon}|\iota(\pi^i)w=0}
  \qquad\text{for } 1\le i\le e(F|\Q_p).
 \]
\end{rmk}

\subsection{The Harder-Narasimhan polygons}\label{S-HN}

Let us recall from \cite{Fa1} the definition of the Harder-Narasimhan polygon of
a $p$-group~$C\in\pgr_{\O_K}$, which describes a certain filtration of $C$ by
finite locally free sub-group-schemes.
After suitable renormalisation,
this will allow us to attach similar invariants to objects of $\pgr_{\O_K,\O_F}$ and,
passing to the limit, of $\pdiv_{\O_K,\O_F}$.

First of all, to each nonzero object~$C\in\pgr_{\O_K}$ is associated a value~$\deg C\in\R$,
called the \emph{degree} of $C$ (see \cite[\S 3]{Fa1}).
Explicitly, if $\omega_C$ has a presentation:
\[
 \omega_C\cong\bigoplus_{i=1}^m\O_K/a_i\O_K
\]
for some $m\in\N$ and nonzero elements $a_1,\dots,a_m\in\O_K$, then:
\[
 \deg C=\sum_{i=1}^m v(a_i)=v(\Fitt_0\omega_C).
\]
The degree is additive in short exact sequences and satisfies the relation~$\deg C+\deg C^D=\height C$.
One defines then the \emph{slope} of $C$ to be $\mi(C)\coloneqq\deg C/\height C\in\R$.
By \cite[\S 4]{Fa1},
the category~$\pgr_{\O_K}$ admits a Harder-Narasimhan formalism for the slope function $\mi$.
More precisely, $C$ is called \emph{semi-stable} if
for every nonzero subobject $C'\subseteq C$ we have $\mi(C')\le\mi(C)$.
In general, $C$ possesses a unique \emph{Harder-Narasimhan filtration}:
\[
 0=C_0\subsetneq C_1\subsetneq\dots\subsetneq C_r=C
\]
in $\pgr_{\O_K}$, with semistable graded pieces and $\mi(C_1)>\dots>\mi(C/C_{r-1})$.
The \emph{Harder-Narasimhan polygon} of $C$ is defined to be:
\[
 \HN(C)\coloneqq
  (\mi(C_1)^{(\height C_1)},\dots,\mi(C/C_{r-1})^{(\height C/C_{r-1})})\in\R^{\height C}_+,
\]
where the superscripts denote the multiplicity of each slope.
Note that $\HN(C)(\height C)=\deg C$.
Let us also remark that the Harder-Narasimhan filtration, and hence the associated polygon,
are invariant under base change of $C$ to the ring of integers of a larger field $K$ within our setup.

Suppose now that $C$ is endowed with endomorphism structure $\iota\colon\O_F\to\End(C)$ for $\O_F$.
Then, by \cite[Corollaire~10]{Fa1}, the Harder-Narasimhan filtration is $\iota$-stable.
We define the Harder-Narasimhan polygon of $(C,\iota)\in\pgr_{\O_K,\O_F}$ to be
the piecewise affine linear function:
\begin{align*}
 \HN(C,\iota)\colon[0,\height C/d] &\longrightarrow\R \\
 x &\longmapsto\frac{1}{d}\HN(C)(dx).
\end{align*}

If $(H,\iota)$ is a $p$-divisible group over $\O_K$ with endomorphism structure for $\O_F$,
we may consider the family of Harder-Narasimhan polygons of
the objects~$(H[p^i],\iota)\in\pgr_{\O_K,\O_F}$, for $i\ge1$.
In order to compare these among each other,
we define the \emph{renormalised} Harder-Narasimhan polygons:
\begin{align*}
 \HN^\r(H[p^i],\iota)\colon[0,n] &\longrightarrow\R \\
 x &\longmapsto\frac{1}{i}\HN(H[p^i],\iota)(ix),
\end{align*}
where $n=\height H/d$ and $i\ge1$.
It follows then from \cite[Théorème~1 and Théorème~7]{Fa3}
that the sequence of functions~$\HN^\r(H[p^i],\iota)$ converges uniformly, for $i\to\infty$,
to a piecewise affine linear, continuous, concave function~$\HN(H,\iota)\colon[0,n]\to\R$,
which is equal to the infimum of the previous functions and
satisfies $\HN(H,\iota)(0)=0$ and $\HN(H,\iota)(n)=\dim H/d$.
We recall that $H$ (equivalently $(H,\iota)$) is called of \emph{HN type} if
the Harder-Narasimhan filtrations of $H[p^i]$, for $i\ge1$,
form a filtration of $H$ by sub-$p$-divisible groups.
By \cite[Proposition~7]{Fa3}, this is equivalent to the fact that $\HN(H,\iota)=\HN(H[p],\iota)$.

Note that for the definition of the above polygons we do not really need to assume that
$K$ contains a Galois closure of $F$ over $\Q_p$, nor that the residue field $\k$ of $\O_K$ is perfect.
All statements concerning the Harder-Narasimhan polygons hold therefore in this more general setup.

\section{The integral Hodge polygon}

\subsection{The definition}

Throughout this whole section,
we let $(H,\iota)$ be a $p$-divisible group over $\O_K$ with endomorphism structure for $\O_F$.
We denote by $(H_\k,\iota)$ the reduction of $H$ to $\k$ together with
the induced $\O_F$-action and write $\height H=\height H_\k=nd$ with $n\in\N$.
We now come to the definition of the integral Hodge polygon of $(H,\iota)$.
Recall that we have a decomposition~\eqref{ctg-dcp-unr} of finite free $\O_K$-modules:
\[
 \omega_H=\bigoplus_\upsilon\omega_{H,\upsilon},
\]
where $\upsilon$ ranges through the $f(F|\Q_p)$ embeddings of $F^\nr$ in $K_0$ and
$\O_{F^\nr}$ acts on $\omega_{H,\upsilon}$ via $\upsilon\colon\O_{F^\nr}\to W(\k)\subseteq\O_K$.
Let $[\pi]\colon\omega_H\to\omega_H$ denote the endomorphism of $\omega_H$ induced by
$\iota(\pi)$ and note that this restricts to each component $\omega_{H,\upsilon}$.
Because $\pi^{e(F|\Q_p)}$ equals $p$ times a unit in $\O_F$ and $\omega_H$ is torsion free,
we have that $[\pi]$ is an injective map.
For every embedding $\upsilon$, then, since $\O_K$ is a valuation ring,
the endomorphism~$[\pi]\colon\omega_{H,\upsilon}\to\omega_{H,\upsilon}$ may be represented,
with respect to suitable $\O_K$-bases, by a diagonal matrix with
nonzero entries~$a_{\upsilon,1},\dots,a_{\upsilon,r_\upsilon}\in\O_K$ with
$v(a_{\upsilon,1})\ge\dots\ge v(a_{\upsilon,r_\upsilon})$,
where $r_\upsilon=\rk_{\O_K}\omega_{H,\upsilon}$;
the elements~$a_{\upsilon,1},\dots,a_{\upsilon,r_\upsilon}$ are uniquely determined up to
units of $\O_K$ and can be found unchanged even with a different choice of uniformiser $\pi\in\O_F$.
Thus, we have an isomorphism of $\O_K$-modules:
\begin{equation}\label{ctg-pi-eld}
 \omega_{H,\upsilon}/[\pi]\omega_{H,\upsilon}\cong\bigoplus_{i=1}^{r_\upsilon}\O_K/a_{\upsilon,i}\O_K.
\end{equation}
We claim that this expression features at most $n$ nontrivial summands.
In order to see this, it suffices to show that the dimension of the $\k$-vector-space given by
the reduction along $\O_K\to\k$ has dimension at most $n$.
Note first that according to the decomposition in \eqref{ctg-dcp-unr-spf},
the identification \eqref{Die-ctg} restricts to:
\begin{equation}\label{Die-ctg-ups}
 \D(H_\k)_\upsilon/\phi_{H_\k}\D(H_\k)_{\sigma^{-1}\upsilon}\cong\omega_{H,\upsilon}\otimes_{\O_K}\k,
\end{equation}
compatibly with the induced $\O_F$-action on both sides.
In particular, $\omega_{H,\upsilon}\otimes_{\O_K}\k$ affords a surjection from $\D(H_\k)_\upsilon$,
which we recall being a free module over $W_{\O_F,\upsilon}(\k)=\O_F\otimes_{\O_{F^\nr},\upsilon}W(\k)$
of rank $\height H_\k/d=n$.
Quotienting by the action of $\pi\in\O_F$, then,
we find that $(\omega_{H,\upsilon}/[\pi]\omega_{H,\upsilon})\otimes_{O_K}\k$ affords a surjection from
a $\k$-vector-space of dimension $n$, proving the claim.
This means that we may always write:
\begin{equation}\label{ctg-pi-eld-red}
 \bar{\omega}_{H,\upsilon}\coloneqq
 \omega_{H,\upsilon}/[\pi]\omega_{H,\upsilon}\cong\bigoplus_{i=1}^n\O_K/a_{\upsilon,i}\O_K
\end{equation}
by neglecting trivial summands from \eqref{ctg-pi-eld} or
possibly setting $a_{\upsilon,i}=1$ for $i>r_\upsilon$.
Let us remark that the last statement remains true even without the assumption that $\k$ is perfect;
in fact, in the argument for the previous claim one may pass to any perfect field extension of $\k$.
We let:
\[
 \Hdg^\int(H,\iota)_\upsilon\coloneqq(v(a_{\upsilon,1}),\dots,v(a_{\upsilon,n}))\in\R^n_+.
\]
Note that the end point of $\Hdg^\int(H,\iota)_\upsilon$ is:
\begin{equation}\label{Hdgi-ups-end}
 \Hdg^\int(H,\iota)_\upsilon(n)=\sum_{i=1}^n v(a_{\upsilon,i})=v(\det([\pi]|\omega_{H,\upsilon}))
 =\frac{1}{e(F|\Q_p)}\rk_{\O_K}\omega_{H,\upsilon},
\end{equation}
the last equality due to the fact that $\pi^{e(F|\Q_p)}$ equals $p$ times a unit in $\O_F$.
More generally, as a piecewise affine linear function,
$\Hdg^\int(H,\iota)_\upsilon\colon[0,n]\to\R$ interpolates the following values:
\[
 \Hdg^\int(H,\iota)_\upsilon(i)=\sum_{j=1}^i v(a_{\upsilon,j})=
 v(\Fitt_0\bar{\omega}_{H,\upsilon})-v(\Fitt_i\bar{\omega}_{H,\upsilon}),
 \quad i\in\Set{0,\dots,n},
\]
where $\Fitt_i\bar{\omega}_{H,\upsilon}$ denotes the $i$-th Fitting ideal (in this case principal)
of the $\O_K$-module~$\bar{\omega}_{H,\upsilon}$.
Observe that $\bar{\omega}_{H,\upsilon}$ is a $p$-torsion $\O_K$-module.
Thus, in the case that the valuation of $K$ is discrete, say of absolute ramification index $e_K$,
we have the following alternative description:
\[
 \Hdg^\int(H,\iota)_\upsilon\colon x\longmapsto\frac{1}{e_K}\sum_{j=1}^{e_K}
  \min\Set{x,\length\bar{\omega}_{H,\upsilon}[\varpi^j]/\bar{\omega}_{H,\upsilon}[\varpi^{j-1}]},
 \quad x\in[0,n],
\]
where $\varpi\in\O_K$ is a uniformiser (so that $v(\varpi)=e_K^{-1}$) and for $0\le j\le e_K$:
\[
 \bar{\omega}_{H,\upsilon}[\varpi^j]\coloneqq\Set{w\in\bar{\omega}_{H,\upsilon}|\varpi^jw=0}.
\]
We define the \emph{integral Hodge polygon} of $(H,\iota)$ to be:
\[
 \Hdg^\int(H,\iota)\coloneqq\frac{1}{f(F|\Q_p)}\sum_\upsilon\Hdg^\int(H,\iota)_\upsilon\in\R^n_+.
\]
Its end point is given by the average over $\upsilon$ of the equations~\eqref{Hdgi-ups-end}:
\[
 \Hdg^\int(H,\iota)(n)=
 \frac{1}{f(F|\Q_p)}\sum_\upsilon\frac{1}{e(F|\Q_p)}\rk_{\O_K}\omega_{H,\upsilon}=
 \frac{1}{d}\rk_{\O_K}\omega_H=\frac{\dim H}{d}.
\]

Note that the above definition is invariant under base change of $(H,\iota)$ to
the valuation ring of a larger field $K$ within our setup.
Moreover, it only makes use of the fact that $\k$ contains $\k_F$,
i.e.\ that $K$ contains $F^\nr$, but not that it contains a whole Galois closure of $F$,
nor that $\k$ is perfect.
Thus, one may define $\Hdg^\int(H,\iota)$ even without these assumptions,
possibly as the integral Hodge polygon of the base change of $(H,\iota)$ to
the valuation ring of a finite field extension of $K$ whose residue field contains $\k_F$.
All statements concerning $\Hdg^\int(H,\iota)$ will then hold in this more general setup.

\begin{rmk}
 Suppose that $F|\Q_p$ is an unramified extension, so that $p\in\O_F$ is a uniformiser.
 In this case,
 we have that $\omega_{H,\upsilon}/[\pi]\omega_{H,\upsilon}=\omega_{H,\upsilon}/p\omega_{H,\upsilon}$
 for every embedding~$\upsilon$ of $F^\nr=F$ in $K_0$.
 In particular, one sees that for $(H,\iota)\in\pdiv_{\O_K,\O_F}$:
 \[
  \Hdg^\int(H,\iota)_\upsilon\coloneqq(1,\dots,1,0,\dots,0),
 \]
 the number of $1$'s being equal to $\rk_{\O_K}\omega_{H,\upsilon}$.
 By the identification~\eqref{Die-ctg-ups},
 this number is also equal to $\dim_\k\D(H_\k)_\upsilon/\phi_{H_\k}\D(H_\k)_{\sigma^{-1}\upsilon}$.
 Thus, recalling the observation in Remark~\ref{Hdg-unr} and taking the average over $\upsilon$,
 we obtain that in this case $\Hdg^\int(H,\iota)=\Hdg(H_\k,\iota)$.
 Putting this together with the considerations of Remark~\eqref{PR-unr},
 we conclude that if $F|\Q_p$ is unramified then $\Hdg^\int(H,\iota)=\Hdg(H_\k,\iota)=\PR(H,\iota)$.
\end{rmk}

\begin{rmk}\label{Hpi}
 Let $H[\pi]$ denote the kernel of the endomorphism~$\iota(\pi)\colon H\to H$,
 a closed sub-group-scheme of $H[p]$.
 Note that $\iota(\pi^{e(F|\Q_p)-1})$ induces an epimorphism $H[p]\to H[\pi]$,
 so that the Hopf algebra of $H[\pi]$ injects into that of $H[p]$ and
 is therefore a torsion free $\O_K$-module;
 since $\O_K$ a valuation ring,
 it follows that $H[\pi]$ is locally free over $\O_K$, i.e.\ it is a $p$-group over $\O_K$.
 The $\O_F$-action on $H$ clearly restricts to $H[\pi]$.
 However, since by definition $\pi$ acts trivially on this object,
 there is no loss of information in considering $H[\pi]$ equipped only with the restriction
 $\iota^\nr\colon\O_{F^\nr}\to\End(H[\pi])$ of the induced action to the unramified part.
 Observe now that we have an $\O_{F^\nr}$-equivariant identification:
 \begin{equation}\label{ctg-pi-Hpi}
  \omega_{H[\pi]}\cong\omega_H/[\pi]\omega_H.
 \end{equation}
 In particular, for $(H,\iota)\in\pdiv_{\O_K,\O_F}$ the polygon~$\Hdg^\int(H,\iota)$ only depends on
 the object~$(H[\pi],\iota^\nr)\in\pgr_{\O_K,\O_{F^\nr}}$.
 In fact, if $F|\Q_p$ is totally ramified,
 then $\Hdg^\int(H,\iota)$ is equal to the Hodge polygon of the $p$-group $H[\pi]$ over $\O_K$,
 as defined by Fargues in \cite[\S 8.2]{Fa1}.
 In general, $\Hdg^\int(H,\iota)$ coincides with the Hodge polygon of the $p$-group with
 unramified endomorphism structure $(H[\pi],\iota^\nr)$ over $\O_K$,
 as considered by Shen in \cite[Definition~3.8]{XS}.
 Here, we really want to view this polygon as an invariant of $(H,\iota)$,
 using it to describe properties of the whole $p$-divisible group with endomorphism structure and
 relating it to other invariants thereof.
\end{rmk}

\subsection{Comparison with the Harder-Narasimhan polygons}

Let us compare $\Hdg^\int(H,\iota)$ with the Harder-Narasimhan polygons recalled in \S\ref{S-HN}.
Consider the object~$(H[\pi],\iota^\nr)\in\pgr_{\O_K,\O_{F^\nr}}$ introduced in Remark~\ref{Hpi}.
More generally, for $i\ge1$ we can consider:
\[
 H[\pi^i]\coloneqq\Ker(\iota(\pi^i)\colon H\to H)
\]
and similarly observe that this is a $p$-group over $\O_K$.
For $1\le j<i$, then, we have exact sequences:
\begin{equation}\label{Hpi-ses}
 0\longrightarrow H[\pi^j]\longrightarrow H[\pi^i]\xlongrightarrow{\iota(\pi)^j}
  H[\pi^{i-j}]\longrightarrow0.
\end{equation}
Since $H[\pi^{e(F|\Q_p)}]=H[p]$,
an inductive reasoning shows that $\height H[\pi]=\height H/e(F|\Q_p)$.
In particular, the function $\HN(H[\pi],\iota^\nr)$ is defined on $[0,n]$ and
can be compared with $\Hdg^\int(H,\iota)$.
In light of Remark~\ref{Hpi},
the following statement really concerns $(H[\pi],\iota^\nr)\in\pgr_{\O_K,\O_{F^\nr}}$ and,
as such, it follows from \cite[Proposition~3.10]{XS}.
We will anyway sketch its proof for completeness.

\begin{prp}[\cite{XS}, Proposition~3.10]\label{Hdgi-Hpi}
 Let $(H,\iota)$ be a $p$-divisible group over $\O_K$ with endomorphism structure for $\O_F$.
 Then:
 \[
  \HN(H[\pi],\iota^\nr)\le\Hdg^\int(H,\iota).
 \]
\end{prp}

\begin{proof}
 First, the $\O_{F^\nr}$-equivariant identification~\eqref{ctg-pi-Hpi} implies:
 \[
  \Hdg^\int(H,\iota)(n)=\frac{1}{f(F|\Q_p)}\deg H[\pi]=\HN(H[\pi],\iota^\nr)(n).
 \]
 Next, let $H'\subseteq H[\pi]$ be a step of the Harder-Narasimhan filtration of $H[\pi]$;
 in particular, $H'$ is $\iota^\nr$-stable.
 We have to check that:
 \[
  \frac{\deg H'}{f(F|\Q_p)}\le\Hdg^\int(H,\iota)\left(\frac{\height H'}{f(F|\Q_p)}\right),
 \]
 i.e.\ that:
 \[
  \deg H'\le\sum_\upsilon\Hdg^\int(H,\iota)_\upsilon\left(\frac{\height H'}{f(F|\Q_p)}\right),
 \]
 where $\upsilon$ ranges through the $f(F|\Q_p)$ embeddings of $F^\nr$ in $K_0$.
 In fact, the inclusion $H'\subseteq H[\pi]$ induces a surjection:
 \[
  \omega_H/[\pi]\omega_H\cong\omega_{H[\pi]}\longrightarrow\omega_{H'},
 \]
 which splits into surjections~$\omega_{H,\upsilon}/[\pi]\omega_{H,\upsilon}\to\omega_{H',\upsilon}$
 for all $\upsilon$'s.
 The result follows then from basic properties of Fitting ideals with respect to exact sequences,
 together with the observation that, for every $\upsilon\colon F^\nr\to K_0$,
 the $\O_K$-module~$\omega_{H',\upsilon}$ is generated by $\height H'/f(F|\Q_p)$ elements
 (see the proof of \cite[Proposition~3.10]{XS}).
\end{proof}

\begin{cor}\label{Hdgi-HN}
 Let $(H,\iota)$ be a $p$-divisible group over $\O_K$ with endomorphism structure for $\O_F$.
 Then:
 \[
  \HN(H[p],\iota)\le\Hdg^\int(H,\iota).
 \]
 In particular:
 \[
  \HN(H,\iota)\le\Hdg^\int(H,\iota).
 \]
\end{cor}

\begin{proof}
 By repeated applications of \cite[Proposition~9]{Fa1} to
 the exact sequences~\eqref{Hpi-ses} with $i=e(F|\Q_p)$,
 we find that $\HN(H[p],\iota)\le\HN(H[\pi],\iota^\nr)$.
\end{proof}

\begin{rmk}\label{Hdgi-HN-rmk}
 The previous corollary provides a sufficient condition for $(H,\iota)$ to be of HN type,
 namely that $\Hdg^\int(H,\iota)=\HN(H,\iota)$.
 Indeed, since $\HN(H[p],\iota)$ is squeezed between these two polygons,
 their equality implies that $\HN(H[p],\iota)=\HN(H,\iota)$.

 More specifically, if a break point of $\HN(H,\iota)$ lies on $\Hdg^\int(H,\iota)$,
 then the previous corollary implies that it also lies on $\HN(H[p],\iota)$.
 This is then sufficient to find an $\iota$-stable sub-$p$-divisible group of $H$ corresponding to
 the break point at hand (see \cite[Corollary~3.4]{AM}, at least if the valuation~$v$ is discrete).
\end{rmk}

\begin{rmk}
 One could look for a relation between the polygons~$\HN(H[p],\iota)$ and $\Hdg(H_\k,\iota)$.
 This is in general hopeless, as the following example shows.
 Assume that $F|\Q_p$ is totally ramified of degree $d=2$.
 Suppose moreover that $\height H=4=2d$, $\dim H=2$ and that, with notation as in \S\ref{S-PR},
 we have $r_\tau=1$ for both embeddings~$\tau\colon F\to K$.
 One can attach to $H$ two ``primitive Hasse invariants'' $m,hasse\in[0,1/2]$,
 such that the usual Hasse invariant of $H$ is given by $ha(H)=2hasse+(p+1)m\in[0,1]$
 (see \cite[\S1]{Bi}).
 Assume that $ha(H)<1/2$ and $p>3$.
 This ensures the existence of the ``canonical subgroup'' of $H[p]$ and
 hence that $\HN(H[p],\iota)(1)=1-ha(H)/2$ (see \cite[Théorème~6]{Fa2}).
 Note that under our assumptions we have $3/4<1-ha(H)/2<1$.
 Now, by \cite[Proposition~1.6]{Bi} we know that $m=0$ if and only if $\omega_H$ is free as
 an $\O_F\otimes_{\Z_p}\O_K$-module.
 Thus,
 if $m=0$ the polygon~$\Hdg(H_\k,\iota)$ has slopes $(1,0)$ and is therefore above $\HN(H[p],\iota)$,
 whereas if $m>0$ then $\Hdg(H_\k,\iota)=(1/2, 1/2)$ is below $\HN(H[p],\iota)$.
\end{rmk}

\subsection{Comparison with the Pappas-Rapoport polygon}

We will now compare the integral Hodge polygon of
$(H,\iota)\in\pdiv_{\O_K,\O_F}$ with its Pappas-Rapoport polygon.
For this purpose, we are going to need the following lemma.

\begin{lem}\label{Hdgi-PR-lem}
 Let $M'\to M$ be an injective homomorphism of finitely presented torsion $\O_K$-modules and
 suppose we are given presentations:
 \begin{equation}\label{pres}
  M\cong\bigoplus_{j=1}^{m}\O_K/a_j\O_K, \qquad M'\cong\bigoplus_{j=1}^{m'}\O_K/b_j\O_K,
 \end{equation}
 for some $m,m'\in\N$ and $a_1,\dots,a_m,b_1,\dots,b_{m'}\in\O_K\setminus\Set{0}$ with
 $v(a_1)\ge\dots\ge v(a_m)$ and $v(b_1)\ge\dots\ge v(b_{m'})$.
 Then, for all $j=1,\dots,\min\Set{m,m'}$ we have that $v(b_j)\le v(a_j)$ and,
 if $m'>m$, then for all $j=m+1,\dots,m'$ we have that $v(b_j)=0$.
\end{lem}

\begin{proof}
 Note first that by possibly adding trivial summands to either of the two presentations,
 we reduce to the case that $m'=m$.

 Let then $e_1,\dots,e_m$ be the standard basis of $M'$ with respect to the given presentation.
 In particular,
 if these elements satisfy a linear equation~$\sum_{j=1}^m c_je_j=0$ for some $c_1,\dots,c_m\in\O_K$,
 then $v(b_j)\le v(c_j)$ for all $j=1,\dots,m$.

 Now, since $M$ is an $a_1$-torsion module and $M'\to M$ is injective,
 we have that $a_1e_1=0$, so that $v(b_1)\le v(a_1)$.
 Fix then $j\in\Set{2,\dots,m}$ and note that the images of $a_je_1,\dots,a_je_j$ in $M$
 are all contained in the submodule $\bigoplus_{l=1}^{j-1}\O_K/a_l\O_K\subseteq M$,
 because the complementary submodule $\bigoplus_{l=j}^m\O_K/a_l\O_K$ is $a_j$-torsion.
 Choosing lifts via the standard surjection~$\O_K^{j-1}\to\bigoplus_{l=1}^{j-1}\O_K/a_l\O_K$,
 we obtain $j$ elements of a free $\O_K$-module of rank $j-1$.
 In particular, these elements must satisfy a nontrivial linear equation,
 where we may assume that at least one coefficient is a unit.
 By projecting this equation back to $\bigoplus_{l=1}^{j-1}\O_K/a_l\O_K$ and
 due to the injectivity of $M'\to M$,
 we find that $\sum_{k=1}^j c_ka_je_k=0$ for some $c_1,\dots,c_j\in\O_K$,
 with $c_{k_0}\in\O_K^\times$ for some $k_0\in\Set{1,\dots,j}$.
 It follows that $v(b_{k_0})\le v(c_{k_0}a_j)=v(a_j)$ and hence that $v(b_j)\le v(b_{k_0})\le v(a_j)$.
 This completes the proof of the lemma.
\end{proof}

We now have all the tools that we need in order to prove the following statement.

\begin{thm}\label{Hdgi-PR}
 Let $(H,\iota)$ be a $p$-divisible group over $\O_K$ with endomorphism structure for $\O_F$.
 Then:
 \[
  \Hdg^\int(H,\iota)\le\PR(H,\iota).
 \]
\end{thm}

\begin{proof}
 First of all, we have that $\Hdg^\int(H,\iota)(n)=\dim H/d=\PR(H,\iota)(n)$.
 Next, recall that we have decompositions~\eqref{ctg-dcp-unr}, respectively \eqref{ctg-dcp}:
 \[
  \omega_H=\bigoplus_{\upsilon\colon F^\nr\to K_0}\omega_{H,\upsilon},\qquad
  \omega_{H,K}\coloneqq\omega_H\otimes_{\O_K}K=\bigoplus_{\tau\colon F\to K}\omega_{H,K,\tau}
 \]
 and that for every $\upsilon\colon F^\nr\to K_0$ the latter equation restricts to
 \eqref{ctg-dcp-ups}:
 \[
  \omega_{H,\upsilon,K}\coloneqq\omega_{H,\upsilon}\otimes_{\O_K}K=
  \bigoplus_{\tau|\upsilon}\omega_{H,K,\tau},
 \]
 where $\tau|\upsilon$ stands for the embeddings of $F$ in $K$ which agree with $\upsilon$ on $F^\nr$.
 By taking the average over the $f(F|\Q_p)$ embeddings~$\upsilon$ of $F^\nr$ in $K_0$,
 it suffices to prove that for each of these $\upsilon$ and $0\le i\le n$ we have:
 \[
  \Hdg^\int(H,\iota)_\upsilon(i)\le\frac{1}{e(F|\Q_p)}\sum_{\tau|\upsilon}\min\Set{i,r_\tau},
 \]
 where $r_\tau=\dim_K\omega_{H,K,\tau}$.
 Fix then $\upsilon\colon F^\nr\to K_0$ and write:
 \[
  \bar{\omega}_{H,\upsilon}\coloneqq
  \omega_{H,\upsilon}/[\pi]\omega_{H,\upsilon}\cong\bigoplus_{i=1}^n\O_K/a_i\O_K
 \]
 with $a_1,\dots,a_n\in\O_K\setminus\Set{0}$, $v(a_1)\ge\dots\ge v(a_n)$, as in \eqref{ctg-pi-eld-red}.
 We need to show that for $0\le i\le n$ and
 any subset~$I\subseteq\Set{\tau\colon F\to K|\tau|\upsilon}$ we have:
 \[
  \sum_{j=1}^i v(a_j)\le\frac{1}{e(F|\Q_p)}\left(\sum_{\tau\notin I}r_\tau+i|I|\right),
 \]
 where $|I|$ denotes the cardinality of $I$.
 So we fix $i$ and $I$ as above and consider the $\O_K$-module:
 \[
  \omega_I\coloneqq\omega_{H,\upsilon}\cap\bigoplus_{\tau\in I}\omega_{H,K,\tau}
  \subseteq\omega_{H,\upsilon,K}.
 \]
 By Lemma~\ref{Hdg-PR-lem}, this is a direct summand of $\omega_{H,\upsilon}$,
 free of rank $r_I\coloneqq\sum_{\tau\in I}r_\tau$.
 Moreover, because $[\pi]\colon\omega_{H,\upsilon}\to\omega_{H,\upsilon}$ restricts to
 an injective endomorphism of $\omega_I$ and $\O_K$ is a valuation ring,
 we may write:
 \[
  \bar{\omega}_I\coloneqq\omega_I/[\pi]\omega_I\cong\bigoplus_{j=1}^{r_I}\O_K/b_j\O_K
 \]
 for some nonzero elements $b_1,\dots,b_{r_I}\in\O_K$ with $v(b_1)\ge\dots\ge v(b_{r_I})$.
 The second assertion of Lemma~\ref{Hdg-PR-lem},
 together with the fact that $[\pi]\omega_{H,\upsilon}\cap\omega_I=[\pi]\omega_I$,
 tells us that $\bar{\omega}_I$ is $\rho_I$-torsion, where $\rho_I=\prod_{\tau\in I}\tau(\pi)\in\O_K$.
 In other words:
 \begin{equation}\label{Hdgi-PR-eqn1}
  v(b_j)\le v(\rho_I)=\frac{|I|}{e(F|\Q_p)}
 \end{equation}
 for all $j=1,\dots,r_I$.
 Note that the inclusion $\omega_I\subseteq\omega_{H,\upsilon}$ induces
 an injective map~$\bar{\omega}_I\to\bar{\omega}_{H,\upsilon}$.
 This, by Lemma~\ref{Hdgi-PR-lem},
 implies that $v(b_j)\le v(a_j)$ for all $j=1,\dots,\min\Set{n,r_I}$ and,
 if $r_I>n$, then $v(b_j)=0$ for all $j=n+1,\dots,r_I$.
 In particular:
 \begin{equation}\label{Hdgi-PR-eqn2}
  \sum_{j=i+1}^{r_I}v(b_j)\le\sum_{j=i+1}^n v(a_j).
 \end{equation}
 Finally, because $\pi^{e(F|\Q_p)}$ equals $p$ times a unit in $\O_F$, we have:
 \[
  \sum_{j=1}^{r_I}v(b_j)=v(\det([\pi]|\omega_I))
  =\frac{1}{e(F|\Q_p)}\rk_{\O_K}\omega_I=\frac{r_I}{e(F|\Q_p)}.
 \]
 This allows us to rewrite:
 \begin{align*}
  \sum_{j=1}^i v(a_j) &=\Hdg^\int(H,\iota)_\upsilon(n)-\sum_{j=i+1}^n v(a_j) \\
  &=\frac{1}{e(F|\Q_p)}\rk_{\O_K}\omega_{H,\upsilon}-\sum_{j=i+1}^n v(a_j) \\
  &=\frac{1}{e(F|\Q_p)}\sum_{\tau\notin I}r_\tau+\sum_{j=1}^{r_I}v(b_j)-\sum_{j=i+1}^n v(a_j).
 \end{align*}
 But using first \eqref{Hdgi-PR-eqn2} and then \eqref{Hdgi-PR-eqn1}, we see that:
 \[
  \sum_{j=1}^{r_I}v(b_j)-\sum_{j=i+1}^n v(a_j)
  \le\sum_{j=1}^{r_I}v(b_j)-\sum_{j=i+1}^{r_I} v(b_j)
  =\sum_{j=1}^{\min\Set{i,r_I}}v(b_j)
  \le\frac{i|I|}{e(F|\Q_p)},
 \]
 thus concluding the proof of the theorem.
\end{proof}

Combining the above statement with Corollary~\ref{Hdgi-HN} we obtain the following inequality,
which generalises \cite[Proposition~3.14]{AM} beyond the case when the valuation~$v$ is discrete
(note that in loc.\ cit.\ the Hodge polygon of an object $(H,\iota)\in\pdiv_{\O_K,\O_F}$ refers to
what we call here the Pappas-Rapoport polygon $\PR(H,\iota)$,
as already observed in Remark~\ref{rmk-PR-nm}).

\begin{cor}
 Let $(H,\iota)$ be a $p$-divisible group over $\O_K$ with endomorphism structure for $\O_F$.
 Then:
 \[
  \HN(H[p],\iota)\le\PR(H,\iota).
 \]
\end{cor}

Consider now the limit case when the integral Hodge polygon is equal to the Pappas-Rapoport polygon.
It turns out that this situation is also detected by the Hodge polygon in special fibre.

\begin{prp}\label{max}
 Let $(H,\iota)$ be a $p$-divisible group over $\O_K$ with endomorphism structure for $\O_F$ and
 denote by $(H_\k,\iota)$ the reduction of $H$ to $\k$ together with the induced $\O_F$-action.
 The equality~$\Hdg^\int(H,\iota)=\PR(H,\iota)$ is equivalent to $\Hdg(H_\k,\iota)=\PR(H,\iota)$.
\end{prp}

\begin{proof}
 Let us recover the same notation as in the previous proof and set:
 \[
  \PR(H,\iota)_\upsilon\coloneqq\frac{1}{e(F|\Q_p)}\sum_{\tau|\upsilon}\PR(H,\iota)_\tau
 \]
 for every embedding~$\upsilon$ of $F^\nr$ in $K_0$.
 Note that the argument for the previous theorem shows that in fact
 $\Hdg^\int(H,\iota)_\upsilon\le\PR(H,\iota)_\upsilon$ for every $\upsilon$.
 Together with Remark~\ref{rmk-gRc}, this allows us to reduce to proving that for every $\upsilon$,
 the equality~$\Hdg^\int(H,\iota)_\upsilon=\PR(H,\iota)_\upsilon$ is equivalent to
 $\Hdg(H_\k,\iota)_\upsilon=\PR(H,\iota)_\upsilon$.
 Fix then $\upsilon\colon F^\nr\to K_0$ and, for $\tau|\upsilon$,
 order the elements $r_\tau=\dim_K\omega_{H,K,\tau}$ as $r_1\ge r_2\ge\dots\ge r_e$,
 where $e\coloneqq e(F|\Q_p)$.
 This gives an ordering~$\tau_1,\dots,\tau_e$ of the set~$\Set{\tau\colon F\to K|\tau|\upsilon}$.
 Recall now from \S\ref{S-Hdg-PR} that
 from this we can define an $\O_F$-stable filtration~\eqref{PR-fil}:
 \begin{equation}\label{PR-fil-max}
  0=\omega_{\upsilon,0}\subseteq\omega_{\upsilon,1}\subseteq\dots\subseteq\omega_{\upsilon,e}
  =\omega_{H,\upsilon}
 \end{equation}
 by $\O_K$-direct-summands,
 with $\rk_{\O_K}\omega_{\upsilon,i}/\omega_{\upsilon,i-1}=r_i$ for $1\le i\le e$ and
 $\O_F$-action acting via $\tau_i$ on each graded piece~$\omega_{\upsilon,i}/\omega_{\upsilon,i-1}$.
 Choosing then an $\O_K$-basis of $\omega_{H,\upsilon}$ adapted to this filtration,
 we may write $[\pi]\colon\omega_{H,\upsilon}\to\omega_{H,\upsilon}$ as a matrix:
 \[
  P=
  \begin{pmatrix}
  \tau_1(\pi)I_{r_1} & A_1 & \dots & * \\
  0 & \tau_2(\pi)I_{r_2} & \ddots & \vdots  \\
  \vdots & \ddots & \ddots & A_{e-1}  \\
  0 & \dots & 0 & \tau_e(\pi)I_{r_e}
  \end{pmatrix},
 \]
 where, for $m\in\N$, $I_m$ denotes the identity matrix of rank $m$.
 Thus, for $1\le i\le e-1$, $A_i$ is a $r_i\times r_{i+1}$-matrix.
 We will prove that both conditions in the statement are equivalent to the fact that
 for all $i=1,\dots,e-1$ the reduction~$\bar{A}_i$ of the matrix~$A_i$ along
 $\O_K\to\k$ has rank $r_{i+1}$.

 Let us first prove the equivalence of the last assertion with
 $\Hdg(H_\k,\iota)_\upsilon=\PR(H,\iota)_\upsilon$.
 Consider the filtration~$0=\omega_{\upsilon,0,\k}\subseteq\omega_{\upsilon,1,\k}\subseteq\dots
 \subseteq\omega_{\upsilon,e,\k}=\omega_{H,\upsilon}\otimes_{\O_K}\k$ obtained as
 the reduction to $\k$ of \eqref{PR-fil-max} and note, recalling the discussion in \S\ref{S-Hdg-PR},
 that the action of $\pi$ induces maps~$\omega_{\upsilon,i+1,\k}/\omega_{\upsilon,i,\k}
 \to\omega_{\upsilon,i,\k}/\omega_{\upsilon,i-1,\k}$ for $1\le i\le e-1$.
 In fact, each of these maps can be represented respectively by the matrix~$\bar{A}_i$.
 Now, $\bar{A}_i$ having rank $r_{i+1}$ for all $i=1,\dots,e-1$ is equivalent to
 all these maps being injective.
 This, in turn, is realised if and only if $\omega_{\upsilon,i,\k}$ is equal to
 the $\iota(\pi)^i$-torsion submodule of $\omega_{H,\upsilon}\otimes_{\O_K}\k$ for all $i=1,\dots,e$.
 However, as observed in Remark~\ref{rmk-gRc},
 this last condition is equivalent to $\Hdg(H_\k,\iota)_\upsilon=\PR(H,\iota)_\upsilon$.

 Assume now that $\Hdg^\int(H,\iota)_\upsilon=\PR(H,\iota)_\upsilon$.
 Looking at the number of nonzero entries of both polygons,
 we find on one side $\rk_{\O_K}\omega_{H,\upsilon}$ minus the rank of
 the reduction~$\bar{P}$ of the matrix~$P$ to $\k$ and on the other side the number $r_1$.
 Since these must be equal and $\rk_{\O_K}\omega_{H,\upsilon}=\sum_{i=1}^er_i$,
 it follows that the rank of $\bar{P}$ is equal to $r_2+\dots+r_e$.
 Now, because the blocks along the diagonal of $P$ reduce to zero,
 this implies that for $1\le i\le e-1$ the rank of $\bar{A}_i$ is maximal, hence equal to $r_{i+1}$.
 Conversely, assume that the rank of $\bar{A}_i$ is equal to $r_{i+1}$ for all $1\le i\le e-1$.
 We will compute the polygon~$\Hdg^\int(H,\iota)_\upsilon$ by induction on $e$,
 where $e$ is seen simply as the number of diagonal blocks of $P$.
 The cases $e=1,2$ are easily done.
 Let then $e\ge3$ and assume that the result is true for $e-1$.
 Let $\varpi\in\O_K$ be an element of valuation $1/e$, so that $v(\varpi)=v(\tau_i(\pi))$ for any $i$.
 Then, applying the inductive hypothesis to the restriction of $[\pi]$ to $\omega_{\upsilon,e-1}$
 (i.e.\ considering the top left submatrix of $P$ containing the first $e-1$ diagonal blocks),
 one can make a change of bases such that the matrix~$P$ transforms to:
 \[
  \begin{pmatrix}
  I_{s_0} & 0 & \dots & 0 & B_0 \\
  0 & \varpi I_{s_1} & \ddots & \vdots & B_1 \\
  \vdots & \ddots & \ddots & 0 & \vdots \\
  \vdots & {} & \ddots & \varpi^{e-1}I_{s_{e-1}} & B_{e-1}  \\
  0 & \dots & \dots & 0 & \tau_e(\pi)I_{r_e}
  \end{pmatrix},
 \]
 where $s_0=r_2+\dots+r_{e-1},s_1=r_1-r_2,\dots,s_{e-2}=r_{e-2}-r_{e-1},s_{e-1}=r_{e-1}$.
 It is also possible to do this in such a way that the reduction of the matrix~$B_{e-1}$ along
 $\O_K\to\k$ has the same rank as $\bar{A}_{e-1}$, that is, $r_e$.
 Up to a further change of bases, one can then reduce to the case where:
 \[
  B_{e-1}=
  \begin{pmatrix}
  I_{r_e} \\
  0
  \end{pmatrix},
  \text{ the zero block having size }(r_{e-1}-r_e)\times r_e,
 \]
 and $B_i = 0$ for $0\le i\le e-2$.
 By an explicit computation,
 one finally finds that $\Hdg^\int(H,\iota)_\upsilon=\PR(H,\iota)_\upsilon$.
\end{proof}

\begin{rmk}
 Concerning a direct comparison between the integral Hodge polygon of
 $(H,\iota)\in\pdiv_{\O_K,\O_F}$ and the Hodge polygon of its reduction to $\k$,
 the above proposition is unluckily the best that one can obtain,
 unless $(H,\iota)$ is ``$\pi$-diagonalisable''
 (see \S\ref{S-pid}, in particular Remark~\ref{Hdg-pid}).
 In general, we will see in \S\ref{S-ex} that the two polygons are unrelated.
\end{rmk}

\subsection{Behaviour with respect to Cartier duality}

Consider the Cartier dual~$H^D$ of the $p$-divisible group $H$ over $\O_K$.
It carries an $\O_F$-action~$\iota^D\colon\O_F\to\End(H^D)$ induced by
$\iota$ and functoriality of Cartier duality.
Then, the integral Hodge polygon of $(H^D,\iota^D)\in\pdiv_{\O_K,\O_F}$ is related to
$\Hdg^\int(H,\iota)$ as in the following statement.

\begin{prp}
 Let $(H,\iota)$ be a $p$-divisible group over $\O_K$ with endomorphism structure for $\O_F$ and
 write $\height H=dn$ with $n\in\N$.
 Let $\upsilon$ be an embedding of $F^\nr$ in $K_0$.
 If $\Hdg^\int(H,\iota)_\upsilon=(\lambda_{\upsilon,1},\dots,\lambda_{\upsilon,n})\in\R^n_+$,
 then $\Hdg^\int(H^D,\iota^D)_\upsilon=(1-\lambda_{\upsilon,n},\dots,1-\lambda_{\upsilon,1})$.
 Thus, if $\Hdg^\int(H,\iota)=(\lambda_1,\dots,\lambda_n)\in\R^n_+$,
 then $\Hdg^\int(H^D,\iota^D)=(1-\lambda_n,\dots,1-\lambda_1)$.
\end{prp}

\begin{proof}
 With notation as in \eqref{ctg-dcp-unr},
 note that the exact sequence~\eqref{Hdg-fil} restricts to
 an $\O_F$-equivariant exact sequence of $\O_K$-modules:
 \[
  0\longrightarrow\omega_{H,\upsilon}\longrightarrow\mathcal{E}_\upsilon
  \longrightarrow\omega_{H^D,\upsilon}^\vee\longrightarrow 0,
 \]
 where $\mathcal{E}_\upsilon$ is a free $O_F\otimes_{O_{F^\nr},\upsilon}O_K$-module of rank $n$ and
 $\omega_{H^D,\upsilon}^\vee=\Hom_{\O_K}(\omega_{H^D,\upsilon},\O_K)$ carries
 the $\O_F$-action induced naturally from that on $\omega_{H^D,\upsilon}$.
 Since $\omega_{H,\upsilon}\cap(\pi\otimes1)\mathcal{E}_\upsilon=[\pi]\omega_{H,\upsilon}$,
 one deduces an exact sequence:
 \[
  0\longrightarrow\omega_{H,\upsilon}/[\pi]\omega_{H,\upsilon}\longrightarrow(\O_K/p\O_K)^n
  \longrightarrow\omega_{H^D,\upsilon}^\vee/[\pi]^\vee\omega_{H^D,\upsilon}^\vee\longrightarrow 0,
 \]
 where $[\pi]^\vee$ denotes the endomorphism of
 $\omega_{H^D,\upsilon}^\vee$ given by the action of $\pi$.
 Observe that if $[\pi]\colon\omega_{H^D,\upsilon}\to\omega_{H^D,\upsilon}$ is represented by
 a certain matrix with respect to two $\O_K$-bases of $\omega_{H^D,\upsilon}$,
 then $[\pi]^\vee$ may be represented by the transpose of this matrix.
 In particular, if we have a presentation of
 $\omega_{H^D,\upsilon}/[\pi]\omega_{H^D,\upsilon}$ as in \eqref{ctg-pi-eld-red},
 which is what defines $\Hdg^\int(H^D,\iota^D)_\upsilon$,
 then we may find the same presentation
 for $\omega_{H^D,\upsilon}^\vee/[\pi]^\vee\omega_{H^D,\upsilon}^\vee$.
 The result can now be deduced by the following consideration.
 Let $N\subseteq\O_K^n$ denote the preimage of $\omega_{H,\upsilon}/[\pi]\omega_{H,\upsilon}$ via
 the standard surjection $\O_K^n\to(\O_K/p\O_K)^n$.
 Then, we have a chain of inclusions~$p\O_K^n\subseteq N\subseteq\O_K^n$ with
 $N/p\O_K^n\cong\omega_{H,\upsilon}/[\pi]\omega_{H,\upsilon}$ and
 $\O_K^n/N\cong\omega_{H^D,\upsilon}^\vee/[\pi]^\vee\omega_{H^D,\upsilon}^\vee$.
\end{proof}

\subsection{Behaviour with respect to subobjects}

Given $(H,\iota)\in\pdiv_{\O_K,\O_F}$ and an $\iota$-stable sub-$p$-divisible group $H'\subseteq H$,
we would like to understand what we can say about the integral Hodge polygon of $H'$
(equipped with the induced $\O_F$-action) in comparison to $\Hdg^\int(H,\iota)$.
For this, we need a dual version of Lemma~\ref{Hdgi-PR-lem}.

\begin{lem}\label{sub-lem}
 Let $M\to M'$ be a surjective homomorphism of finitely presented torsion $\O_K$-modules and
 suppose we are given presentations as in \eqref{pres}.
 Then, the same conclusions as in Lemma~\ref{Hdgi-PR-lem} hold.
\end{lem}

\begin{proof}
 For any finitely presented torsion $\O_K$-module~$N$, we define a new $\O_K$-module:
 \[
  N^*\coloneqq\Hom_{\O_K}(N,K/\O_K).
 \]
 Note that if $N=\O_K/a\O_K$ with $a\in\O_K$ nonzero
 (i.e.\ if $N$ is cyclic), then $N^*\cong\O_K/a\O_K$,
 the element $1+a\O_K\in\O_K/a\O_K$ corresponding to the map defined by $1+a\O_K\mapsto a^{-1}+\O_K$.
 In general, if $N$ has a certain presentation as direct sum of cyclic modules,
 then we can present $N^*$ in the same way.

 Now, our surjective map~$M\to M'$ gives rise to an injection~$M'^*\to M^*$.
 Since, as observed above,
 $M'^*$ and $M^*$ can be presented in the same way as $M'$ and $M$ respectively,
 we are then reduced to the statement of Lemma~\ref{Hdgi-PR-lem}.
\end{proof}

Suppose now that we have an $\iota$-stable sub-$p$-divisible group $H'$ of $H$.
Write $\height H'=dn'$ with $n'\in\N$ and let $\iota'$ denote the restriction of $\iota$ to $H'$.
The inclusion~$H'\subseteq H$ induces
an $\O_F$-equivariant surjective map~$\omega_H\to\omega_{H'}$.
According to the decomposition~\eqref{ctg-dcp-unr}, this map splits as a direct sum of
surjections~$\omega_{H,\upsilon}\to\omega_{H',\upsilon}$,
with $\upsilon$ ranging through the embeddings of $F^\nr$ in $K_0$.
Quotienting by the action of $\pi$,
we obtain a surjective map~$\bar{\omega}_{H,\upsilon}\to\bar{\omega}_{H',\upsilon}$ for
each $\upsilon$, with notation as in \eqref{ctg-pi-eld-red}.
Applying now the previous lemma, we deduce that for every $i=1,\dots,n'$ the slope of
$\Hdg^\int(H',\iota')_\upsilon$ on the interval~$[i-1,i]$ is less than or equal to the slope of
$\Hdg^\int(H,\iota)_\upsilon$ on the same interval.
Taking the average over $\upsilon$, we find a similar statement for
the full integral Hodge polygons~$\Hdg^\int(H',\iota')$ and $\Hdg^\int(H,\iota)$.
In particular,
the end point~$(n',\dim H'/d)$ of $\Hdg^\int(H',\iota')$ lies on or below $\Hdg^\int(H,\iota)$.
In fact, if this point lies on $\Hdg^\int(H,\iota)$,
then $\Hdg^\int(H',\iota')$ is forced to be equal to
the restriction of $\Hdg^\int(H,\iota)$ to $[0,n']$.
Note that this is for instance the case if $H'$ is the subobject corresponding to a break point of
$\HN(H,\iota)$ lying on $\Hdg^\int(H,\iota)$, as in Remark~\ref{Hdgi-HN-rmk}.

Let us sum up the above conclusions in the following proposition.

\begin{prp}
 Let $(H,\iota)$ be a $p$-divisible group over $\O_K$ with endomorphism structure for $\O_F$,
 let $H'\subseteq H$ be an $\iota$-stable sub-$p$-divisible group,
 write $\height H'=dn'$ with $n'\in\N$ and let $\iota'$ denote the restriction of $\iota$ to $H'$.
 Then, for every $i=1,\dots,n'$ the slope of $\Hdg^\int(H',\iota')$ on
 $[i-1,i]$ is at most the slope of $\Hdg^\int(H,\iota)$ on the same interval.
 Moreover, if $(n',\dim H'/d)$ lies on $\Hdg^\int(H,\iota)$,
 then $\Hdg^\int(H',\iota')$ is equal to the restriction of $\Hdg^\int(H,\iota)$ to $[0,n']$.
\end{prp}

\begin{rmk}
 Consider the situation of the proposition just stated and let $H''$ be the quotient of $H$ by $H'$;
 it has an induced $\O_F$-action, which we denote by $\iota''$.
 Then, arguing by duality,
 we see that for every $i=1,\dots,n-n'$ the slope of $\Hdg^\int(H'',\iota'')$ on
 $[i-1,i]$ is at least the slope of $\Hdg^\int(H,\iota)$ on $[n'+i-1,n'+i]$.
 Moreover, if $(n',\dim H'/d)$ lies on $\Hdg^\int(H,\iota)$,
 then $\Hdg^\int(H'',\iota'')$ is equal to the restriction of $\Hdg^\int(H,\iota)$ to $[n',n]$,
 up to a shift of coordinates setting the origin in $(n',\dim H'/d)$.
\end{rmk}

\subsection{The \texorpdfstring{$\pi$}{π}-diagonalisable case}\label{S-pid}

In this section we are going to study a special case of $p$-divisible groups over
$\O_K$ with endomorphism structure for $\O_F$, defined as follows.

\begin{dfn}
 The object~$(H,\iota)\in\pdiv_{\O_K,\O_F}$ is called \emph{$\pi$-diagonalisable} if
 the endomorphism~$[\pi]$ of the free $\O_K$-module~$\omega_H$ is diagonalisable.
\end{dfn}

\begin{rmk}\label{rmk-pid}
 Recall that $K$ is assumed to contain a Galois closure of $F$ over $\Q_p$.
 \begin{enumerate}
  \item If $[\pi]\colon\omega_H\to\omega_H$ is diagonalisable,
  then the eigenspace decomposition of $\omega_H$ must recover \eqref{ctg-dcp} after inverting $p$,
  hence each eigenspace of $\omega_H$ is contained in a corresponding
  $\omega_{H,K,\tau}\subseteq\omega_H\otimes_{\O_K}K$.
  In particular, we deduce that $[\pi]\colon\omega_H\to\omega_H$ is diagonalisable if and
  only if the decomposition~\eqref{ctg-dcp} restricts to:
  \begin{equation}\label{ctg-dcp-pid}
   \omega_H=\bigoplus_\tau(\omega_H\cap\omega_{H,K,\tau})
  \end{equation}
  (in general we only have a containment relationship).
  This also means that for $(H,\iota)\in\pdiv_{\O_K,\O_F}$ the notion of
  being $\pi$-diagonalisable does not actually depend on the choice of the uniformiser~$\pi\in\O_F$.

  \item If $F|\Q_p$ is unramified,
  then every object $(H,\iota)\in\pdiv_{\O_K,\O_F}$ is $\pi$-diagonalisable.
  In general, using \eqref{ctg-dcp-unr},
  the $\pi$-diagonalisability reduces to a condition on the restrictions of $[\pi]$ to
  $\omega_{H,\upsilon}$ for each $\upsilon\colon F^\nr\to K_0$.
  Observe now that if $[\pi]\colon\omega_H\to\omega_H$ is diagonalisable,
  then its eigenvalues belong to $\Set{\tau(\pi)\in\O_K|\tau\colon F\to K}$,
  so they are all elements of valuation $1/e(F|\Q_p)$.
  Thus, if $(H,\iota)$ is $\pi$-diagonalisable,
  then for every $\upsilon\colon F^\nr\to K_0$ we have:
  \[
   \Hdg^\int(H,\iota)_\upsilon=\left(\frac{1}{e(F|\Q_p)},\dots,\frac{1}{e(F|\Q_p)},0,\dots,0\right),
  \]
  the number of nonzero entries being equal to $\rk_{\O_K}\omega_{H,\upsilon}$;
  in particular, $\rk_{\O_K}\omega_{H,\upsilon}\le n=\height H/d$ for each $\upsilon$.
  In this sense, the polygons $\Hdg^\int(H,\iota)_\upsilon$ give a measure of the obstruction to
  the diagonalisability of the endomorphism $[\pi]\colon\omega_H\to\omega_H$.

  \item A converse to the statement above is false in general.
  Let for instance $F=\Q_p(\pi)$ with $\pi^2=p=2$.
  Consider then a $p$-divisible group $H$ over $\O_K$ of height~$4$ and dimension~$2$,
  endowed with an $\O_F$-action~$\iota\colon\O_F\to\End(H)$ such that
  the endomorphism~$[\pi]$ of $\omega_H$ induced by $\iota(\pi)$ is given,
  with respect to some $\O_K$-basis~$(e_1,e_2)$, by a matrix:
  \[
  \begin{pmatrix}
   \sqrt{2} & c \\
   0 & -\sqrt{2}
  \end{pmatrix}
  \]
  with $c\in\O_K$ satisfying $1/2\le v(c)<3/2$.
  Now, because $[\pi](e_1)=\sqrt{2}e_1$ and $[\pi](e_2)=-\sqrt{2}(e_2-c/\sqrt{2}e_1)$,
  we have that $\Hdg^\int(H,\iota)=(1/2,1/2)$.
  However, if we denote by $\tau_+$ and $\tau_-$ the two embeddings of $F$ in $K$,
  defined by $\tau_+(\pi)=\sqrt{2}$ and $\tau_-(\pi)=-\sqrt{2}$,
  then we see that $\omega_H\cap\omega_{H,K,\tau_+}=\O_K\cdot e_1$ and
  $\omega_H\cap\omega_{H,K,\tau_-}=\O_K\cdot(e_1-2\sqrt{2}/ce_2)$,
  which together do not span the whole $\omega_H$.
  Note that in this example $F|\Q_p$ is a wildly ramified extension.
  It turns out that if we rule out this possibility,
  then the properties deduced in part (2) of this remark characterise
  the $\pi$-diagonalisability completely.
 \end{enumerate}
\end{rmk}

\begin{prp}\label{pid}
 Assume that $F|\Q_p$ is tamely ramified and
 let $(H,\iota)$ be a $p$-divisible group over $\O_K$ with endomorphism structure for $\O_F$.
 Write $\height H=dn$ with $n\in\N$ and
 consider the decomposition~$\omega_H=\bigoplus_\upsilon\omega_{H,\upsilon}$ as in ~\eqref{ctg-dcp-unr}.
 Then, $(H,\iota)$ is $\pi$-diagonalisable if and only if for every $\upsilon$ we have that
 $\rk_{\O_K}\omega_{H,\upsilon}\le n$ and:
 \begin{equation}\label{Hdgi-pid}
  \Hdg^\int(H,\iota)_\upsilon=\left(\frac{1}{e(F|\Q_p)},\dots,\frac{1}{e(F|\Q_p)},0,\dots,0\right),
 \end{equation}
 the number of nonzero entries being equal to $\rk_{\O_K}\omega_{H,\upsilon}$.
\end{prp}

\begin{proof}
 One direction of the double implication is the content of part (2) of the remark above.
 For the other direction, assume that for every embedding~$\upsilon\colon F^\nr\to K_0$
 we have that $r_\upsilon\coloneqq\rk_{\O_K}\omega_{H,\upsilon}\le n$ and
 $\Hdg^\int(H,\iota)_\upsilon$ satisfies the formula in the statement,
 with as many $1/e(F|\Q_p)$ entries as $r_\upsilon$.
 By the first part of Remark~\ref{rmk-pid}, it suffices to prove that \eqref{ctg-dcp} restricts to
 $\omega_H=\bigoplus_\tau(\omega_H\cap\omega_{H,K,\tau})$ or, equivalently,
 that for each $\upsilon$ we have:
 \[
  \omega_{H,\upsilon}=\bigoplus_{\tau|\upsilon}(\omega_H\cap\omega_{H,K,\tau}),
 \]
 where $\tau|\upsilon$ denotes the embeddings~$\tau\colon F\to K$ which agree with
 $\upsilon$ on $F^\nr$.
 Fix then $\upsilon\colon F^\nr\to K_0$ and note that for every $\tau|\upsilon$,
 by Lemma~\ref{Hdg-PR-lem} the $\O_K$-module~$\omega_H\cap\omega_{H,K,\tau}$ is a direct summand of
 $\omega_{H,\upsilon}$, free of rank $\dim_K\omega_{H,K,\tau}$.
 We set:
 \[
  \tilde{\omega}_{H,\tau}\coloneqq(\omega_H\cap\omega_{H,K,\tau})\otimes_{\O_K}\k
   \subseteq\omega_{H,\upsilon}\otimes_{\O_K}\k.
 \]
 By Nakayama's lemma, it is enough to show that $\omega_{H,\upsilon}\otimes_{\O_K}\k$ is equal to
 the sum over $\tau|\upsilon$ of the sub-$\k$-vector-spaces~$\tilde{\omega}_{H,\tau}$.
 However, since:
 \[
  \sum_{\tau|\upsilon}\dim_\k\tilde{\omega}_{H,\tau}=\sum_{\tau|\upsilon}\dim_K\omega_{H,K,\tau}=
  r_\upsilon=\dim_\k(\omega_{H,\upsilon}\otimes_{\O_K}\k),
 \]
 the claim is equivalent to the fact that, for varying $\tau|\upsilon$,
 the $\tilde{\omega}_{H,\tau}$'s have pairwise intersection $0$.

 In order to prove this, recall that because $F|\Q_p$ is tamely ramified,
 we may choose the uniformiser~$\pi$ of $\O_F$ in such a way that $\pi^{e(F|\Q_p)}\in F^\nr$.
 Then, the elements~$\tau(\pi)$ for $\tau|\upsilon$ are exactly
 the different $e(F|\Q_p)$-th roots of $\upsilon(\pi^{e(F|\Q_p)})$ in $K$.
 Thus, fixing one embedding~$\tau_0\colon F\to K$ with $\tau_0|\upsilon$,
 we may write $\Set{\tau\colon F\to K|\tau|\upsilon}=\{\tau_0,\dots,\tau_{e(F|\Q_p)-1}\}$ with
 $\tau_i(\pi)=\zeta^i\tau_0(\pi)$, where $\zeta\in\O_K$ is a primitive $e(F|\Q_p)$-th root of $1$.
 Now, by assumption the endomorphism~$[\pi]\colon\omega_{H,\upsilon}\to\omega_{H,\upsilon}$
 may be represented, with respect to suitable $\O_K$-bases, by a diagonal matrix with
 nonzero entries $a_1,\dots,a_{r_\upsilon}\in\O_K$ all of valuation $1/e(F|\Q_p)$.
 By rescaling the elements of one of the two bases by the appropriate units of $\O_K$,
 we may even represent $[\pi]$ by a scalar matrix of coefficient any element $a\in\O_K$ with
 $v(a)=1/e(F|\Q_p)$; let us take $a=\tau_0(\pi)$.
 This allows us to write $[\pi]=\tau_0(\pi)\psi$ with
 $\psi$ an invertible endomorphism of $\omega_{H,\upsilon}$.
 Let $\bar{\psi}$ denote the reduction of $\psi$ to $\omega_{H,\upsilon}\otimes_{\O_K}\k$.
 Note at this point that for every $i=0,\dots,e(F|\Q_p)-1$,
 the restriction of $[\pi]$ to $\omega_H\cap\omega_{H,K,\tau_i}\subseteq\omega_{H,\upsilon}$
 coincides with the multiplication by $\tau_i(\pi)$,
 hence the restriction of $\psi$ to the same submodule
 coincides with the multiplication by $\tau_i(\pi)/\tau_0(\pi)=\zeta^i$.
 This, in turn, implies that the restriction of $\bar{\psi}$ to $\tilde{\omega}_{H,\tau}$
 equals the multiplication by the image~$\bar{\zeta^i}$ of $\zeta^i$ in $\k$.
 However, since $e(F|\Q_p)$ is coprime to $p$,
 the elements~$\bar{\zeta^i}\in\k$ are distinct for $i=0,\dots,e(F|\Q_p)-1$, ensuring that
 the sub-vector-spaces~$\tilde{\omega}_{H,\tau_0},\dots,\tilde{\omega}_{H,\tau_{e(F|\Q_p)-1}}$ have pairwise intersection $0$.
 This concludes the proof of the proposition.
\end{proof}

\begin{rmk}\label{Hdg-pid}
 Write $\omega_H=\bigoplus_\upsilon\omega_{H,\upsilon}$ as in \eqref{ctg-dcp-unr}.
 Let then $\upsilon$ be an embedding of $F^\nr$ in $K_0$ and
 suppose that $\rk_{\O_K}\omega_{H,\upsilon}\le n$.
 We claim that if $\Hdg^\int(H,\iota)_\upsilon$ satisfies the formula~\eqref{Hdgi-pid},
 then the same holds for $\Hdg(H_\k,\iota)_\upsilon$.
 In particular, if $(H,\iota)\in\pdiv_{\O_K,\O_F}$ is $\pi$-diagonalisable,
 then taking into consideration the second part of Remark~\ref{rmk-pid},
 we have that for every $\upsilon\colon F^\nr\to K_0$:
 \[
  \Hdg(H_\k,\iota)_\upsilon=\left(\frac{1}{e(F|\Q_p)},\dots,\frac{1}{e(F|\Q_p)},0,\dots,0\right)=
   \Hdg^\int(H,\iota)_\upsilon
 \]
 and ultimately $\Hdg(H_\k,\iota)=\Hdg^\int(H,\iota)$.

 To prove the claim, recall first the $\O_F$-equivariant identification~\eqref{Die-ctg-ups}:
 \[
  \D(H_\k)_\upsilon/\phi_{H_\k}\D(H_\k)_{\sigma^{-1}\upsilon}\cong\omega_{H,\upsilon}\otimes_{\O_K}\k,
 \]
 where the left-hand side is part of the decomposition~\eqref{ctg-dcp-unr-spf}.
 Observe next, as in the previous proof, that by assumption
 (and up to base changing to the ring of integers of a suitable finite field extension of $K$)
 the restriction of
 $[\pi]\colon\omega_H\to\omega_H$ to $\omega_{H,\upsilon}$ can be expressed as
 the multiplication by an element of $O_K$ of valuation $1/e(F|\Q_p)$,
 composed with an invertible endomorphism.
 Now, the reduction of this map to $\omega_{H,\upsilon}\otimes_{\O_K}\k$ is zero.
 Thus, given the above identification, we have that
 $\D(H_\k)_\upsilon/\phi_{H_\k}\D(H_\k)_{\sigma^{-1}\upsilon}$ is $\pi\otimes1$-torsion as
 a module over $W_{\O_F,\upsilon}(\k)=\O_F\otimes_{\O_{F^\nr},\upsilon}W(\k)$,
 and this proves the claim.

 The converse to the statement which we just showed is false in general.
 Consider again the case $F=\Q_p(\pi)$ with $\pi^2=p$ (for any $p$)
 and assume that $K$ contains an element~$c$ with $0<v(c)<1/2$.
 Let then $(H,\iota)\in\pdiv_{\O_K,\O_F}$ with $\height H=4$,
 $\dim H=2$ and such that $[\pi]\colon\omega_H\to\omega_H$ is given,
 with respect to some $\O_K$-basis~$(e_1,e_2)$, by a matrix:
 \[
 \begin{pmatrix}
  \sqrt{p} & c \\
  0 & -\sqrt{p}
 \end{pmatrix}.
 \]
 Then, because the reduction of this matrix to $\k$ is zero,
 we have that $\Hdg(H_\k,\iota)=(1/2,1/2)$.
 However, over $\O_K$ we see that
 $[\pi](e_2)=c(e_1-\sqrt{p}/ce_2)$ and $[\pi](e_1-\sqrt{p}/ce_2)=p/ce_2$,
 hence $\Hdg^\int(H,\iota)=(1-v(c),v(c))\ne(1/2,1/2)$.
\end{rmk}

\paragraph{The case of $\O_F$-modules.}
 As examples of $\pi$-diagonalisable $p$-divisible groups over
 $\O_K$ with endomorphism structure for $\O_F$,
 we should mention a particularly relevant class of objects defined as follows.
 Fix an embedding $\tau_0\colon F\to K$.
 The object $(H,\iota)\in\pdiv_{\O_K,\O_F}$ is called a $p$-divisible \emph{$\O_F$-module} if
 the $\O_F$-action induced by $\iota$ on $\omega_H$ is via $\tau_0\colon\O_F\to\O_K$.
 In this case, the decomposition~\eqref{ctg-dcp} reduces to $\omega_{H,K}=\omega_{H,K,\tau_0}$,
 so that the equality in \eqref{ctg-dcp-pid} is clearly satisfied.
 In fact, the previous formulas simplify considerably for $p$-divisible $\O_F$-modules.
 First of all, we have that $\dim H\le n=\height H/d$ by Remark~\ref{ctg-dcp-dim}.
 Furthermore, the polygons~$\Hdg^\int(H,\iota)_\upsilon$ are zero for
 all embeddings~$\upsilon\colon F^\nr\to K_0$ except for
 the one given by the restriction of $\tau_0$.
 Thus:
 \[
  \Hdg^\int(H,\iota)=
   \frac{1}{f(F|\Q_p)}\left(\frac{1}{e(F|\Q_p)},\dots,\frac{1}{e(F|\Q_p},0,\dots,0\right)=
   \left(\frac{1}{d},\dots,\frac{1}{d},0,\dots,0\right),
 \]
 the number of nonzero entries being equal to $\dim H$.
 By Remark~\ref{Hdg-pid},
 a similar situation holds for the polygons $\Hdg(H_\k,\iota)_\upsilon$ and $\Hdg(H_\k,\iota)$.
 Finally, note that even the Pappas-Rapoport polygon of $(H,\iota)$ has only slopes $1/d$ and $0$,
 so that in the end:
 \[
  \Hdg^\int(H,\iota)=\Hdg(H_\k,\iota)=\PR(H,\iota)=
   \left(\frac{1}{d},\dots,\frac{1}{d},0,\dots,0\right).
 \]
 In particular, $p$-divisible $\O_F$-modules are also an example of
 the special case considered in Proposition~\ref{max}.

\section{The integral Hodge polygon in families}

\subsection{Continuity of the integral Hodge polygon}\label{S-cont}

In this section,
we let $E$ be a complete valued field extension of $\Q_p$ and denote by $\O_E$ its valuation ring.
We assume that $E$ contains $F^\nr$, although we do not fix an embedding.
Let $\varpi\in\O_E$ be a pseudouniformiser.
For $r,s\in\N$,
we denote by $\O_E\{S_1,\dots,S_r\}$ the $\varpi$-adic completion of $\O_E[S_1,\dots,S_r]$ and,
if $E$ is discretely valued,
we denote by $\O_E\{S_1,\dots,S_r\}[[T_1,\dots,T_s]]$ the $(\varpi,T_1,\dots,T_s)$-adic completion of
$\O_E[S_1,\dots,S_r,T_1,\dots,T_s]$.
Recall that these rings have a description in terms of (converging) power series,
in particular they do not depend on the choice of $\varpi$.

Let now $\X$ be a formal scheme over $\Spf{\O_E}$ and consider the following two situations.
\begin{enumerate}[label=(\Alph*)]
 \item $\X$ is \emph{locally finitely presented} over $\O_E$,
  i.e.\ it is a locally finite union of open affine formal subschemes, say $\Spf A_i$,
  with each $A_i$ a topologically finitely presented $\O_E$-algebra, that is,
  an $\O_E$-algebra of the form $\O_E\{S_1,\dots,S_r\}/I$ for some $r\in\N$ and
  $I\subseteq\O_E\{S_1,\dots,S_r\}$ a finitely generated ideal.
 \item $E$ is discretely valued and $\X$ is a \emph{special} formal scheme over $\Spf\O_E$,
  i.e.\ it is a locally finite union of open affine formal subschemes, say $\Spf A_i$,
  with each $A_i$ a special $\O_E$-algebra, that is,
  an $\O_E$-algebra of the form $\O_E\{S_1,\dots,S_r\}[[T_1,\dots,T_s]]/I$ for some $r,s\in\N$ and
  $I\subseteq\O_E\{S_1,\dots,S_r\}[[T_1,\dots,T_s]]$ a finitely generated ideal.
  This notion can also be found in the literature under the name of
  \emph{locally formally of finite type} (without the locally finiteness assumption).
\end{enumerate}
We let $\X^\an$ denote the generic fibre of $\X$ as a Berkovich $E$-analytic space
(see \cite[\S 1]{Be2} for the construction in case (A) and \cite[\S 1]{Be3} in case (B)).
Let us remark that the topological space underlying $\X^\an$ is Hausdorff and paracompact.

Observe that in both situations considered above, the formal scheme $\X$ is adic.
If $\I\subseteq\O_\X$ is an ideal of definition of finite type,
where $\O_\X$ is the structure sheaf of $\X$,
we let $\X_i$ denote the scheme with the same underlying topological space as $\X$ and
structure sheaf $\O_\X/\I^{i+1}$, for $i\ge0$.
Then, $\X$ is equal to the inductive limit of the schemes~$\X_i$,
with respect to the obvious transition maps.
By a $p$-divisible group over $\X$ we shall mean a family~$\H=(\H_i)_{i\ge0}$ where
each $\H_i$ is a $p$-divisible group over $\X_i$,
together with compatible isomorphisms~$\H_j\times_{\X_j}\X_i\cong\H_i$ for $j\ge i$.
We say that $\H$ has constant height~$h\in\N$ if each $\H_i$ has constant height $h$ over $\X_i$
(equivalently, if $\H_0$ has constant height~$h$ over $\X_0$).
A homomorphism of $p$-divisible groups over $\X$ is given by a compatible system of
homomorphisms of $p$-divisible groups over $\X_i$, for $i\ge0$.
We remark that these notions are independent of the choice of the ideal of definition $\I$.

\begin{dfn}
 A \emph{$p$-divisible group} over $\X$ \emph{with endomorphism structure} for
 $\O_F$ is a pair~$(\H,\iota)$ consisting of a $p$-divisible group~$\H$ over $\X$ and
 a map of $\Z_p$-algebras $\iota\colon\O_F\to\End(\H)$.
 Note that if $\I\subseteq\O_\X$ is an ideal of definition of finite type and
 $\H=(\H_i)_{i\ge0}$ with notation as above, then $\iota$ amounts to a compatible system of
 maps of $\Z_p$-algebras $\iota_i\colon\O_F\to\End(\H_i)$, for $i\ge0$.
\end{dfn}

Let then $(\H,\iota)$ be a $p$-divisible group over $\X$ with endomorphism structure for $\O_F$ and
suppose that $\H$ has constant height, say $h\in\N$.
For each point $x\in\X^\an$, we let $\Hr(x)$ denote the completed residue field of $x$;
it is a complete valued field extension of $E$.
Then, $x$ induces a map of formal schemes $\Spf\O_x\to\X$ over $\Spf\O_E$,
where $\O_x$ is the valuation ring of $\Hr(x)$.
By base change of $(\H,\iota)$,
we obtain a $p$-divisible group over $\Spf\O_x$ with endomorphism structure for $\O_F$.
Recall that this amounts to the same thing as a $p$-divisible group over $\Spec\O_x$ with
endomorphism structure for $\O_F$ (see \cite[Lemma~II.4.16]{Me});
we denote this object by $(\H_x,\iota)\in\pdiv_{\O_x,\O_F}$.
Note that in particular we have that $h=\height\H_x=dn$, for some $n\in\N$ uniform on $\X$.
Consider now the integral Hodge polygon~$\Hdg^\int(\H_x,\iota)\in\R^n_+$ of $(\H_x,\iota)$.
This defines a function:
\begin{align*}
 \Hdg^\int(\H,\iota)\colon\X^\an &\longrightarrow\R^n_+ \\
 x &\longmapsto\Hdg^\int(\H_x,\iota).
\end{align*}
Let us endow $\R^n_+$ with the topology induced by the usual topology of $\R$.
Then, we have the following result.

\begin{thm}\label{cont}
 Let $\X$ be a formal scheme over $\Spf\O_E$ falling in situation (A) or (B) as above and
 denote by $\X^\an$ its generic fibre as a Berkovich $E$-analytic space.
 Let $(\H,\iota)$ be a $p$-divisible group over $\X$ with endomorphism structure for $\O_F$,
 suppose that $\H$ has constant height $h$ and write $h=dn$ with $n\in\N$.
 Then, the function~$\Hdg^\int(\H,\iota)\colon\X^\an\longrightarrow\R^n_+$ is continuous.
 Moreover, for every fixed element $f_0\in\R^n_+$, the subset:
 \[
  \X^\an_{\Hdg^\int\le f_0}\coloneqq\Set{x\in\X^\an|\Hdg^\int(\H_x,\iota)\le f_0}\subseteq\X^\an
 \]
 defines a closed analytic domain of $\X^\an$.
\end{thm}

\begin{proof}
 Let $\I\subseteq\O_\X$ be an ideal of definition of finite type and
 write $\X=\varinjlim_{i\ge0}\X_i$ and $\H=(\H_i)_{i\ge0}$ as above.
 For $i\ge0$, let $\omega_{\H_i}$ denote the cotangent sheaf of $\H_i$ along the identity section.
 Then, $\omega_\H\coloneqq\varprojlim_{i\ge0}\omega_{\H_i}$ is a finite and locally free $\O_\X$-module
 (see \cite[3.3.1]{BBM} and recall that $\X$ is $p$-adically complete),
 with the property that for each point~$x\in\X^\an$,
 the pullback of $\omega_\H$ along $\Spf\O_x\to\X$ identifies with $\omega_{\H_x}$.
 Note that $\omega_\H$ is endowed with an action of $\O_F$ induced by $\iota$.
 Thus, because $E$ contains $F^\nr$,
 we have a decomposition~$\omega_\H=\bigoplus_\upsilon\omega_{\H,\upsilon}$ into
 finite locally free $\O_\X$-submodules,
 where $\upsilon$ ranges through the embeddings of $F^\nr$ in $E$ and
 $\O_{F^\nr}$ acts on $\omega_{\H,\upsilon}$ via $\upsilon\colon\O_{F^\nr}\to\O_E\to\O_\X$.
 For each $\upsilon$, let $[\pi]\colon\omega_{\H,\upsilon}\to\omega_{\H,\upsilon}$ denote
 the endomorphism induced by $\iota(\pi)$ and
 set $\bar{\omega}_{\H,\upsilon}\coloneqq\omega_{\H,\upsilon}/[\pi]\omega_{\H,\upsilon}$,
 a coherent $\O_\X$-module.
 For $j\in\N$, then, we have coherent sheaves of ideals $\Fitt_j\bar{\omega}_{\H,\upsilon}$ which
 locally recover the Fitting ideals of $\bar{\omega}_{\H,\upsilon}$ (see \cite[2.8.13]{Ab}).
 We obtain a description of the function~$\Hdg^\int(\H,\iota)$ as follows.
 Let $x\in\X^\an$, write $v_x$ for the valuation on $\Hr(x)$ and,
 for $j\ge0$ and $\upsilon\colon F^\nr\to E$,
 denote by $(\Fitt_j\bar{\omega}_{\H,\upsilon})_x$ the pullback of
 $\Fitt_j\bar{\omega}_{\H,\upsilon}$ along $\Spf\O_x\to\X$.
 Then:
 \begin{align*}
  \Hdg^\int(\H,\iota)(x)(j)
   &=\frac{1}{f(F|\Q_p)}\sum_\upsilon\Hdg^\int(\H_x,\iota)_\upsilon(j) \\
   &=\frac{1}{f(F|\Q_p)}\sum_\upsilon
    (v_x((\Fitt_0\bar{\omega}_{\H,\upsilon})_x)-v_x((\Fitt_j\bar{\omega}_{\H,\upsilon})_x)) \\
   &=\Hdg^\int(\H_x,\iota)(n)-
    \frac{1}{f(F|\Q_p)}\sum_\upsilon v_x((\Fitt_j\bar{\omega}_{\H,\upsilon})_x) \\
   &=\frac{\dim\H_x}{d}-\frac{1}{f(F|\Q_p)}v_x\left(\prod_\upsilon(\Fitt_j\bar{\omega}_{\H,\upsilon})_x\right),
 \quad j\in\Set{0,\dots,n}.
 \end{align*}
 Recall now, from the construction of $\X^\an$,
 that if $\X=\bigcup_k\X_k$ is a locally finite covering of $\X$ by open affine formal subschemes,
 then the generic fibres $\X_k^\an$ provide a locally finite covering of
 $\X^\an$ by closed analytic domains,
 so that both claimed statements may be checked on each $\X_k^\an$ independently.
 Thus, it suffices to treat the case that $\X$ is affine, say $\X=\Spf A$,
 with $A$ an $\O_E$-algebra which is topologically of finite type (in situation (A)),
 respectively special (in situation (B)).
 Furthermore, we may assume without loss of generality that $\omega_\H$ has constant rank over $\X$,
 hence that $x\mapsto\dim\H_x$ is constant over $\X^\an$.
 Let now $j\in\Set{0,\dots,n}$ and choose generators $f_{j,1},\dots,f_{j,r_j}\in A$ of
 $\prod_\upsilon\Fitt_j\bar{\omega}_{\H,\upsilon}$;
 observe that these elements define regular functions on $\X^\an$.
 The results follow then from the fact that for every $x\in\X^\an$ we have:
 \[
  v_x\left(\prod_\upsilon(\Fitt_j\bar{\omega}_{\H,\upsilon})_x\right)
   =\min\Set{v_x(f_{j,1}(x)),\dots,v_x(f_{j,r_j}(x))}.
 \]
\end{proof}

\begin{rmk}
 Assume that $F|\Q_p$ is totally ramified and that we are in situation (B).
 Then, in light of Remark~\ref{Hpi}, the above theorem also follows from \cite[Proposition~16]{Fa1}.
\end{rmk}

\begin{rmk}
 Let us outline the meaning of the previous theorem
 first from the point of view of classical rigid analytic geometry,
 then in terms of adic spaces.
 \begin{enumerate}
 \item
 Let $\X^\rig$ denote the generic fibre of $\X$ as a rigid analytic space over $E$
 (see \cite[\S 0.2]{Bert}).
 We can identify $\X^\rig$ with the rigid analytic space associated to
 the Berkovich space~$\X^\an$ (see \cite[\S 1.6]{Be1});
 as a set, it is given by the points~$x\in\X^\an$ such that $\Hr(x)$ is a finite extension of $E$.
 Moreover, the intersection with $\X^\rig$ of any analytic domain of $\X^\an$
 is an open admissible subset of $\X^\rig$.
 Thus, the final statement of Theorem~\ref{cont} means that for every fixed element $f_0\in\R^n_+$,
 the subset~$\X^\rig_{\Hdg^\int\le f_0}\coloneqq\X^\rig\cap\X^\an_{\Hdg^\int\le f_0}\subseteq\X^\rig$
 is open admissible.
 In fact, $\X^\rig$ may be constructed even without the locally finiteness assumption in (A) or (B),
 with any open affine covering of $\X$ yielding an admissible covering of $\X^\rig$.
 Furthermore, every point $x\in\X^\rig$ induces again a map of
 formal schemes $\Spf\O_x\to\X$ over $\Spf\O_E$,
 so that we can define the function~$\Hdg^\int(\H,\iota)$ similarly as above.
 Then, the above statement holds even in this generality,
 with an analogous proof to that of the final statement of Theorem~\ref{cont}.

 \item
 Let $\X^\ad$ denote the adic space associated with $\X^\an$ (see \cite[Remark~8.3.2]{Hu}).
 As topological spaces, we can identify $\X^\an$ with the maximal Hausdorff quotient of $\X^\ad$
 (see loc.\ cit.\ Lemma~8.1.8).
 Then, by composing $\Hdg^\int(\H,\iota)\colon\X^\an\to\R^n_+$ with the quotient map~$\X^\ad\to\X^\an$,
 we obtain by Theorem~\ref{cont} a continuous function~$\Hdg^\int(\H,\iota)\colon\X^\ad\to\R^n_+$.
 Note that, $\R^n_+$ being a Hausdorff topological space,
 it is really equivalent whether we consider $\Hdg^\int(\H,\iota)$ as
 a function from $\X^\ad$ or from $\X^\an$.
 \end{enumerate}
\end{rmk}

\subsection{An example}\label{S-ex}

Let us study here a noteworthy example in the framework of the previous discussion.
Assume that $E$ is discretely valued,
that it contains a square root $\sqrt{p}$ of $p$ and that its residue field~$\k_E$ is perfect.
Let then $F=\Q_p(\pi)$ with $\pi^3=p$, so that $d=e(F|\Q_p)=3$.
We consider the special $\O_E$-algebra~$A=\O_E[[T,U]]/(TU-\sqrt{p})$ and
the affine formal scheme~$\X=\Spf A$ over $\Spf\O_E$.
Then, $\X^\an$ is an open annulus inside the $E$-analytic closed unit disc~$D_E(1)$:
\[
 \X^\an=\Set{x\in D_E(1)|0<v_x(T)<\frac{1}{2}}\subseteq D_E(1)=(\Spf\O_E\{T\})^\an.
\]
We choose the ideal of definition $\I\subseteq\O_\X$ generated by $T,U$ and a uniformiser of $E$,
so that, with notation as before, $\X_0=\Spec\k_E$ corresponds to the reduced special fibre of $\X$.
Let now $H$ be the $p$-divisible group over $\X_0$ whose contravariant Dieudonné crystal
(in the sense of \cite[\S 3.3]{BBM}) evaluated at $W(\k_E)$ is given by:
\[
 D_H=W(\k_E)^6,\qquad
 \phi_H=
 \begin{pmatrix}
 0  & p  & 0  & {} & {} & {} \\
 0  & 0  & p  & {} & \text{\huge{$0$}} & {} \\
 1  & 0  & 0  & {} & {} & {} \\
 {} & {} & {} & 0  & 0  & p  \\
 {} & \text{\huge{$0$}} & {} & 1  & 0  & 0  \\
 {} & {} & {} & 0  & 1  & 0
 \end{pmatrix}\sigma.
\]
Here, $W(\k_E)$ is the ring of Witt vectors with coefficients in $\k_E$ and
$\sigma$ its Frobenius endomorphism.
We let $\pi\in\O_F$ act on $(D_H,\phi_H)$ via the matrix:
\[
 [\pi]=
 \begin{pmatrix}
 0  & 0  & p  & {} & {} & {} \\
 1  & 0  & 0  & {} & \text{\huge{$0$}} & {} \\
 0  & 1  & 0  & {} & {} & {} \\
 {} & {} & {} & 0  & 0  & p  \\
 {} & \text{\huge{$0$}} & {} & 1  & 0  & 0  \\
 {} & {} & {} & 0  & 1  & 0
 \end{pmatrix},
\]
which satisfies $[\pi]\circ\phi_H=\phi_H\circ[\pi]$ and $[\pi]^3=p\cdot\id_{D_H}$.
Denoting by $\iota$ the corresponding $\O_F$-action on $H$,
we obtain an object $(H,\iota)\in\pdiv_{\k_E,\O_F}$, with $\height H=6=2d$.

Consider now the free $A$-module~$D_H\otimes_{W(\k_E)}A=A^6$,
denote by $f_1,\dots,f_6$ the standard basis and let $L$ be the free submodule generated by
$e_1\coloneqq f_2+Tf_4$, $e_2\coloneqq f_3+Tf_5$, $e_3\coloneqq U\sqrt{p}f_1+f_6$.
Observe that the quotient $A^6/L$ is a free $A$-module and that
$L\otimes_A A/\I\subseteq A^6\otimes_A A/\I=D_H/p D_H$ identifies with
the sub-$\k_E$-vector-space~$V_H D_H/p D_H$,
where $V_H$ denotes the Verschiebung of $(D_H,\phi_H)$, i.e.\ $V_H\phi_H=p\cdot\id_{D_H}$.
Then, by Grothendieck-Messing theory,
there exists a $p$-divisible group $\H=(\H_i)_{i\ge0}$ over $\X$ with $\H_0=H$ and $\omega_\H=L$,
where $\omega_\H=\varprojlim_{i\ge0}\omega_{\H_i}$ as in the proof of the previous theorem
(here, we apply \cite[Theorem~V.1.6]{Me} inductively to the thickenings $\X_{i+1}\to\X_i$, for $i\ge0$;
since the defining ideal has square zero,
these thickenings are endowed with the trivial divided power structure).
Note that $L\subseteq D_H\otimes_{W(\k_E)}A$ is stable under the action of $\pi$ that
we obtain extending $[\pi]$ by linearity.
Precisely, the restriction of $[\pi]$ to $L$ is given by the matrix:
\begin{equation}\label{exA}
 \begin{pmatrix}
 0 & 0 & U\sqrt{p} \\
 1 & 0 & 0 \\
 0 & T & 0
 \end{pmatrix}
\end{equation}
with respect to the basis $e_1,e_2,e_3$.
In particular, $\iota$ upgrades to an $\O_F$-action on $\H$,
giving rise to a $p$-divisible group~$(\H,\iota)$ over $\X$ with endomorphism structure for $\O_F$.
Moreover, the induced action of $\pi$ on $\omega_\H=L$ is described by the above matrix.
Using this, we can then compute the function:
\begin{align*}
 \Hdg^\int(\H,\iota)\colon\X^\an &\longrightarrow\R^2_+ \\
 x &\longmapsto\Hdg^\int(\H_x,\iota)=(v_x(U\sqrt{p}),v_x(T))=(1-v_x(T),v_x(T)).
\end{align*}
Observe that for each point~$x\in\X^\an$,
corresponding to the map of formal schemes $\Spf\O_x\to\X$ over $\Spf\O_E$,
the reduction of $(\H_x,\iota)$ to the residue field $\k_x$ of $\O_x$ coincides with
the base change of $(H,\iota)$ along $\k_E\to\k_x$.
In particular, its Hodge polygon is given by:
\[
 \Hdg(H,\iota)=\left(\frac{2}{3},\frac{1}{3}\right),
\]
as it can be seen reducing \eqref{exA} modulo $\I$.
The same observation can be made with regard to the ``Newton polygon'',
for whose definition in this context we refer to \cite[définition~3.1 and Proposition~1.1.9]{BH1}
(it is in fact just a rescaled version of the classical Newton polygon,
see loc.\ cit.\ Remarque~1.1.12).
In our case, up to reversing the order of the slopes to comply with our conventions,
the Newton polygon of $(H,\iota)$, hence of the reduction of $(\H_x,\iota)$ to $\k_x$, is:
\[
 \Newt(H,\iota)=\left(\frac{2}{3},\frac{1}{3}\right).
\]
In conclusion, the family $(\H,\iota)$ provides examples of objects in $\pdiv_{\O_K,\O_F}$
(for suitable complete valued fields $K|E$) whose integral Hodge polygon can be either
above the Hodge polygon and the Newton polygon of their reduction
(for $x\in\X^\an$ with $0<v_x(T)\le1/3$) or below them (for $x\in\X^\an$ with $1/3\le v_x(T)<1/2$).

\noindent
\textsc{Stéphane Bijakowski} \\
Centre de Mathématiques Laurent Schwartz (CMLS), CNRS, \'Ecole polytechnique,
Institut Polytechnique de Paris, 91120 Palaiseau, France \\
E-mail: \href{mailto:stephane.bijakowski@polytechnique.edu}{stephane.bijakowski@polytechnique.edu}

\bigskip\noindent
\textsc{Andrea Marrama} \\
Dipartimento di Matematica ``Tullio Levi-Civita'', Università degli Studi di Padova \\
Via Trieste 63, 35121 Padova, Italy \\
E-mail: \href{mailto:andrea.marrama@math.unipd.it}{andrea.marrama@math.unipd.it}


\begin{thebibliography}{99}
 \bibitem[Ab]{Ab} A. Abbes --
  \emph{Éléments de Géométrie Rigide, I: Construction et étude géométrique des espaces rigides};
  Progress in Mathematics 286, Birkhäuser, 2010.
 \bibitem[Be1]{Be1} V. G. Berkovich --
  \emph{Étale cohomology for non-Archimedean analytic spaces};
  Publications Mathématiques de l’IHES 78 : 5-161, 1993.
 \bibitem[Be2]{Be2} V. G. Berkovich --
  \emph{Vanishing cycles for formal schemes};
  Inventiones mathematicae 115 : 539-571, 1994.
 \bibitem[Be3]{Be3} V. G. Berkovich --
  \emph{Vanishing cycles for formal schemes. II};
  Inventiones mathematicae 125 : 367-390, 1996.
 \bibitem[Bert]{Bert} P. Berthelot --
  \emph{Cohomologie rigide et cohomologie rigide à supports propres};
  Prépublication IRMAR 96-03, 1996.
 \bibitem[BBM]{BBM} P. Berthelot, L. Breen, W. Messing --
  \emph{Théorie de Dieudonné Cristalline II};
  Lecture Notes in Mathematics 930, Springer, 1982.
 \bibitem[Bi]{Bi} S. Bijakowski --
  \emph{Partial Hasse Invariants, Partial Degrees, and the Canonical Subgroup};
  Canadian Journal of Mathematics 70 (4) : 742-772, 2018.
 \bibitem[BH1]{BH1} S. Bijakowski, V. Hernandez --
  \emph{Groupes $p$-divisibles avec condition de Pappas-Rapoport et invariants de Hasse};
  Journal de l’École polytechnique - Mathématiques 4 : 935-972, 2017.
 \bibitem[BH2]{BH2} S. Bijakowski, V. Hernandez --
  \emph{On the geometry of the Pappas-Rapoport models for PEL Shimura varieties};
  Journal of the Institute of Mathematics of Jussieu 22 (5) : 2403-2445, 2023.
 \bibitem[Fa1]{Fa1} L. Fargues --
  \emph{La filtration de Harder-Narasimhan des schémas en groupes finis et plats};
  Journal für die reine und angewandte Mathematik 645 : 1-39, 2010.
 \bibitem[Fa2]{Fa2} L. Fargues --
  \emph{La filtration canonique des points de torsion des groupes $p$-divisibles};
  Annales scientifiques de l'École Normale Supérieure 4\textsuperscript{e} série 44 (6) : 905-961,
  2011.
 \bibitem[Fa3]{Fa3} L. Fargues --
  \emph{Théorie de la réduction pour les groupes $p$-divisibles};
  preprint, \href{https://arxiv.org/abs/1901.08270}{arXiv:1901.08270}.
 \bibitem[Fo]{Fo1} J.-M. Fontaine --
  \emph{Groupes $p$-divisibles sur les corps locaux};
  Astérisque 47-48, 1977.
 \bibitem[Hu]{Hu} R. Huber --
  \emph{Étale Cohomology of Rigid Analytic Varieties and Adic Spaces};
  Aspects of Mathematics 30, Vieweg+Teubner, 1996.
 \bibitem[Ma]{AM} A. Marrama --
  \emph{Hodge-Newton filtration for $p$-divisible groups with ramified endomorphism structure};
  Documenta Mathematica 27 : 1805-1863, 2022.
 \bibitem[Me]{Me} W. Messing --
  \emph{The Crystals Associated to Barsotti-Tate Groups: With Applications to Abelian Schemes};
  Lecture Notes in Mathematics 264, Springer, 1972.
 \bibitem[RR]{RR} M. Rapoport, M. Richartz --
  \emph{On the classification and specialization of $F$-isocrystals with additional structure};
  Compositio Mathematica 103 (2) : 153-181, 1996.
 \bibitem[Sh]{XS} X. Shen --
  \emph{On the Hodge-Newton Filtration for $p$-Divisible Groups with Additional Structures};
  International Mathematics Research Notices 2014 (13) : 3582-3631, 2014.
\end{thebibliography}
\end{document}